\documentclass[11pt,a4paper]{article}
\usepackage{color}
\usepackage{multirow}
\usepackage{caption}
\usepackage[leqno]{amsmath}
\usepackage{enumitem}
\usepackage{ragged2e,amsfonts,amsthm}
\usepackage{graphicx,epstopdf}
\usepackage{caption,subcaption}
\RequirePackage{xcolor}[2.11]
\usepackage[colorlinks=true, linkcolor=blue,
    		   citecolor=red!60!black, urlcolor=blue]{hyperref}
\usepackage[noabbrev,capitalise,nameinlink]{cleveref}
\usepackage{float, setspace}

\usepackage[ruled,vlined]{algorithm2e}
\usepackage{etoolbox}
\makeatletter
\patchcmd{\@algocf@start}
  {-1.5em}
  {0pt}
  {}{}
\makeatother

\newtheorem{theorem}{Theorem}[section]
\newtheorem{lemma}{Lemma}[section]
\newtheorem{prop}{Proposition}[section]
\newtheorem{corollary}{Corollary}[section]

\newtheorem{assumption}{Assumption}[section]
\newtheorem{remark}{Remark}[section]
\newtheorem{setup}{Setup}[section]

\numberwithin{equation}{section}

\def\B{{\mathbb B }}
\def\R{{\mathbb R }}

\def\bea{\vspace{-3mm}
         \begin{eqnarray}
		\arraycolsep=2pt\def\arraystretch{1.5}
		\begin{array}}
\def\eda{\end{array}\vspace{-3mm}\end{eqnarray}}
\def\nnb{\nonumber}

 \providecommand{\tightlist}{%
      \setlength{\itemsep}{0pt}
      \setlength{\parskip}{4pt}}

\graphicspath{ {./images/} }

\title{\vspace{-1cm} An Oracle Gradient Regularized Newton Method for\\ Quadratic Measurements Regression}

\author{Jun Fan$^*$, Jie Sun$^{\dag*}$, Ailing Yan$^*$ and Shenglong Zhou$^\ddag$\\
Institute of Mathematics, Hebei University of Technology, China$^*$\\
School of EECMS, Curtin University, Australia and \\School of Business, National University of Singapore, Singapore$^\dag$\\
School of Mathematics and Statistics, Beijing Jiaotong University, China$^\ddag$
\thanks{Corresponding author: Shenglong Zhou}}

\date{}

\begin{document}
\flushbottom

\maketitle

 \vspace{-12mm}

\begin{abstract}
\noindent {\bf Abstract:}   
Recovering an unknown signal from quadratic measurements has gained popularity due to its wide range of applications, including phase retrieval, fusion frame phase retrieval, and positive operator-valued measures. In this paper, we employ a least squares approach to reconstruct the signal and establish its non-asymptotic statistical properties. Our analysis shows that the estimator perfectly recovers the true signal in the noiseless case, while the error between the estimator and the true signal is bounded by $O(\sqrt{p\log(1+2n)/n})$ in the noisy case, where  $n$ is the number of measurements and 
$p$ is the dimension of the signal. {We then develop a two-phase algorithm, gradient regularized Newton method (GRNM), to solve the least squares problem. It is proven that the first phase terminates within finitely many steps, and the sequence generated in the second phase converges to a unique local minimum at a superlinear rate under certain mild conditions.} Beyond these deterministic results, GRNM is capable of exactly reconstructing the true signal in the noiseless case and achieving the stated error rate with a high probability in the noisy case. Numerical experiments demonstrate that GRNM offers a high level of recovery capability and accuracy as well as fast computational speed.

\vspace{0.3cm}

\vspace{0.3cm}

\noindent{\bf MR(2000) Subject Classification}~ 90C26, 90C30, 90C90

\end{abstract}

\section{Introduction }
Phase retrieval has a wide range of applications in X-ray crystallography, transmission electron microscopy, and coherent diffractive imaging, see \cite{candesetal,candes13,eldaretal15,KSTX15LAA,she-beck-eldar14}   and the references therein. It
aims at recovering the lost phase information through the observed magnitudes.  More recently, \cite{wang-xu19} and \cite{hrwx21} extended the idea of phase retrieval to a more generalized setting to reconstruct a vector $x\in\mathbb{R}^p$ or $\mathbb{C}^p$ from quadratic samples $x^HA_1x,\ldots , x^HA_nx$, where $\{A_1,\ldots,A_n\}$ is a set of Hermitian matrices and $^H$ stands for the Hermitian transpose. The generalized case is capable of including the standard phase retrieval and its various spinoffs as special cases. For example, it  reduces to the standard phase retrieval if $A_i=a_ia_i^H$ with $a_i\in\mathbb{C}^p$, and coincides with the fusion frame
(projection) phase retrieval if
all $A_i$ are orthogonal projections, namely $A_i^2=A_i$. Moreover, if all $A_i$ are positive semi-definite matrices
satisfying $\sum_{i=1}^nA_i=I$, where $I$ is the identity matrix, then it becomes the positive operator valued measure, which is an active research topic in quantum tomography (see e.g. \cite{hmw13}). This generalized case can be described as the following quadratic measurements regression (QMR, \cite{fanetal18}) problem,
\begin{equation}
b_i=\langle x^*, A_ix^*\rangle+{\varepsilon}_i,\quad{} {i\in[n]:=\{1,\ldots ,n\}},\tag{QMR} \label{PQMR}
\end{equation}
where $b\in\mathbb{R}^n$ and $A_i\in \mathbb{S}^{p\times{}p},  {i\in[n]}$ are the observed vector from the real space and measurement matrices from symmetrical matrix space $\mathbb{S}^{p\times{}p}$ respectively, $x^*\in\mathbb{R}^p$ is unknown true signal to be reconstructed, and ${\varepsilon}_i\in\mathbb{R}$ is a random noise with zero mean and variance $\mathbb{E}\varepsilon_i^2$.  Here,   $\langle\cdot,\cdot\rangle$ represents the vector inner product. We permit the variance of ${\varepsilon}_i$ equals to zero for each $i\in[n]$, which is called the the noiseless version of \eqref{PQMR}.

We would like to point out that, here and below, we only focus on  \eqref{PQMR} in real space as the complex case can be reformulated as a real one. In fact, if $A_i \in \mathbb{C}^{p\times p}$ and $x^*\in\mathbb{C}^p$, then let
\begin{eqnarray}
u^*:=\left[  \begin{array}{cr}  \mathcal{R}(x^*)  \\
\mathcal{I}(x^*)
      \end{array}\right],\qquad ~~~~M_i:=\left[  \begin{array}{cr}    \mathcal{R}(A_i) &-\mathcal{I}(A_i)\\
     \mathcal{I}(A_i)&\mathcal{R}(A_i)
      \end{array}\right],\nnb
\end{eqnarray}
where $\mathcal{R}(\cdot)$ and $\mathcal{I}(\cdot)$ denote the real and imaginary part for a complex number respectively.
  Then \eqref{PQMR} can be rewritten as
$$b_i=\langle x^*, A_ix^*\rangle +\varepsilon_i=\langle u^*, M_iu^*\rangle +\varepsilon_i, \qquad{}  i\in[n].$$
Since $\mathcal{R}(A_i)$ is symmetric and $\mathcal{I}(A_i)$ is antisymmetric,  $M_i$ is symmetric. A popular approach to recover signal $x^*$ from (\ref{PQMR}) is to solve the following least squares problem,
\bea{c}\label{LSE-PQMR}
\hat x:=\underset{x\in\mathbb{R}^p}{\rm argmin}~ f(x):=\frac{1}{4n}\sum_{i=1}^n(\langle x, A_ix\rangle-b_i)^2,
\eda
 which is the focal point of our analysis. 

\subsection{Related work}\label{sec:related works}
{\bf (a) Theoretical analysis.}  In terms of the recovery performance of model  (\ref{PQMR}), most attention is paid to the phase retrieval problems. The earlier work can be traced back to \cite{bandeira}, in which authors studied some necessary and sufficient conditions for measurement vectors to yield injective and stable intensity measurements $\{a_1,\ldots, a_n\}$. In \cite{sun18}, it was shown that if the measurements are identically and independently distributed (i.i.d.) complex Gaussian and the number of measurements $n$ is large enough, with a high probability,  there are no spurious local minimizers of model \eqref{LSE-PQMR} and all global minimizers are equal to the target signal. In  \cite{huang20},  a sharp estimation error of the model is achieved without assuming any distributions on the noise if the measurement vectors are random  Gaussian.

For the work on the general quadratic measurements model (\ref{PQMR}),  the authors in \cite{wang-xu19}  introduced the concept of phase retrieval property. They explored the connections among
generalized phase retrieval, low-rank matrix recovery, and nonsingular bilinear form. They presented results on the minimal measurement number needed for recovering a matrix. Furthermore, they applied these results to the phase retrieval problems and showed that  $p\times{}p$ order matrices $\{A_1,\ldots , A_n\}$ have the phase retrieval property
if $n\geq2p-1$ in the real case and $n\geq4p-4$ in the complex case. Subsequently in  \cite{hrwx21}, the almost everywhere phase retrieval property has been investigated for a set of matrices $\{A_1,\ldots , A_n\}$. In  \cite{Thaker}, the authors studied the quadratic feasibility problem and established conditions under which the problem becomes identifiable. Then they adopted model \eqref{LSE-PQMR} to precess this problem and claimed that any gradient-based algorithms from an arbitrary initialization can converge to a globally optimal point with a high probability if matrices {$\{A_1,\ldots , A_n\}$ are Hermitian matrices sampled from a complex Gaussian distribution.

{\bf (b) Algorithmic development.}  When it comes to solving model (\ref{PQMR}), most numerical algorithms can be categorized into two groups. The first group requires a good initial point for better recovery performance. For instance, the Wirtinger flow (WF) method was first developed in \cite{candeswf15}. It combines a good initial guess obtained by the spectral method and a series of updates that refine the initial estimate by a gradient descent scheme in the sense
of Wirtinger calculus iteratively. It is proven that the WF method converges to a true signal at a linear rate for Gaussian random measurements. The authors in \cite{ma2018} proposed a modified Levenberg-Marquardt method with an initial point provably close to the set of the global optimal solutions and established the global linear convergence and local quadratic convergence properties.   Other prior work that took advantage of spectral initialization or its variant to achieve the convergence rate can be found in \cite{chen17,wang17,zhang17,machen2018,MM2018,CFL2018,
lmcc2019,ll2020,gao2020}.

The second group of algorithms for the phase retrieval problems focuses on establishing the global convergence to the true signal starting from arbitrary initial points.  For instance, a two-stage algorithm was proposed in \cite{gao17}, with the first stage aiming to obtain a good initialization for the second stage, where an optimization problem was solved by the Gauss-Newton method.  It was shown in \cite{cai2022} that the gradient descent methods with any random initial point can converge to the true signal and may outperform the state-of-the-art algorithms with spectral initialization in empirical success rate and convergence
speed. Similarly, the authors in \cite{chenmp19}  pointed out that
compared with spectral methods, random initialization is model-agnostic and usually more robust vis-a-vis model mismatch.  Other candidates include the rust-region method \cite{sun18} and alternating minimization/projection methods \cite{zhang2020}.




For the general quadratic measurements model, there is very little work except for an algorithm in \cite{hgd19, hgd20}, which comprises a spectral initialization, followed by iterative gradient descent updates. It was proven that when the number of measurements  is large enough, a global optimal solution can be recovered from complex quadratic measurements with a high probability. Other than that, to the best of our knowledge, no other viable numerical algorithm has been proposed for solving the general quadratic measurements model. Moreover, we note that almost all aforementioned algorithms for phase retrieval problems and the algorithm of general quadratic measurements need a common assumption that requires the entries of the measurement matrices to be i.i.d. Gaussian and noiseless. One of the motivations of this paper is to design a direct second-order method for solving the general quadratic measurements model that starts  from  arbitrary initial points, has good  numerical performance and  its convergence properties do not depend on the assumption on Gaussian measurements.

\subsection{Contributions}
The main contributions of this paper are fourfold.
\begin{itemize}\tightlist
\item[I.] We study the non-asymptotic reconstruction performance of model \eqref{LSE-PQMR} for the original problem \eqref{PQMR}. It is shown that the global minimum $\hat{x}$ to the model  exactly recovers the true signal $x^*$ for the noiseless case and that
\bea{l}\min\{\|\hat{x}-x^*\|,\|\hat{x}+x^*\|\} =O(\sqrt{p\log(1+2n)/n})\nnb\eda
 for the noisy case, see Theorem \ref{nonasym-bound}. The results above are achieved under  Assumption \ref{sample-RIP-general} on $n$ matrices $\{A_1,\ldots , A_n\}$ that is guaranteed with a high probability by setting up those matrices as sub-Gaussian matrices, see Lemma \ref{sample-RIP}. Therefore, our results are obtained by sub-Gaussian measurements rather than Gaussian measurements.
\item[II.] To solve optimization model \eqref{LSE-PQMR}, we design a two-phase algorithm comprising of the gradient descent and regularized Newton methods (GRNM). Under the isometry property-like assumption, the algorithm converges to a stationary point $\bar{x}$ (i.e.,  $\nabla{}f( \bar{x})=0$) of \eqref{LSE-PQMR} from any initial points. Under additional mild conditions, the algorithm will generate a sequence { eventually} converging to a unique local minimum to \eqref{LSE-PQMR} at convergence rate $1+\delta$, where $\delta\in(0,1)$ (see Theorem \ref{convergence-rate}). All of these results are achieved from the deterministic perspective.
\item[III.] We then establish the oracle convergence results for GRNM from the statistical perspective. For  noiseless \eqref{PQMR}, as proven in Theorem \ref{oracle-noiseless},  GRNM is capable of capturing the true signal with   convergence rate of $1+\delta$. While for the noisy \eqref{PQMR}, its generated solution reaches the error rate of $O(\sqrt{p\log(1+2n)/n})$ with a high probability (see Theorem \ref{oracle-conv-noise}) and the convergence is superlinear in probability (see Theorem \ref{oracle-conv-suplinear-noise}).
\item[IV.] It is demonstrated that GRNM has a nice numerical performance in terms of the recovery accuracy and capability. Its computational speed is more favorable  compared with several leading solvers \cite{hgd19, fann20, hrwx21}. In addition, unlike the Newton-type algorithms, our method does not rely on the initial points heavily.
\end{itemize}

\subsection{Paper outline}

The rest of this paper is organized as follows. In the next section, we present some notation and assumptions  and study the properties of the objective function defined in \eqref{LSE-PQMR}. In Section \ref{sec:sta-ana}, we conduct the non-asymptotic statistical analysis to show the reconstruction performance of model \eqref{LSE-PQMR} for the original problem \eqref{PQMR} in both noiseless and noisy cases. The gradient regularized Newton method (GRNM) is developed in  Section \ref{sec:alg} where the convergence analysis from the deterministic perspective is also provided. In Section \ref{sec:oracle-con}, we investigate the oracle convergence of GRNM to see the reconstruction performance for \eqref{PQMR} in both noiseless and noisy cases, followed by a report on numerical experiments in Section \ref{sec:num}. We conclude the paper   in the last section.

\section{Preliminaries}

We start with presenting some notation that will be used throughout the paper. For a scalar $a$, let   $\lfloor a \rfloor$  be the largest integer that is no more than $a$.  Denote  $\|\cdot\|$ the Euclidean norm for a vector and the spectral norm for a matrix, while the Frobenius norm is denoted by $\|\cdot\|_F$.  We write the unit ball as $\B:=\{x\in\R^p:\|x\|\leq1\},$
  the identity matrix as $I$,  and a positive semi-definite (definite) matrix $A$ as $A\succeq0$ ($A\succ0$), respectively.
We denote $ {\lambda}_{\max}(A)  $ and $ {\lambda}_{\min}(A) $ the largest eigenvalue and smallest eigenvalues of matrix $A$. Let
\bea{rcl}
\label{phi-Ax}
\varphi_i(x)&:=& \langle x, A_ix\rangle-b_i,~ i\in[n],\\
{\bf A}(x)&:=&(A_1x,A_2x,\ldots,A_mx)^\top\in\R^{n\times p}.
\eda
Then  the gradient and Hessian of $f$  in \eqref{LSE-PQMR} are as follows
\bea{rcl}
\label{grad-hess}
\nabla f(x)&=&\frac{1}{n} \sum_{i=1}^n\varphi_i(x)A_ix,\\
\nabla^2f(x)&=&\frac{1}{n}\sum_{i=1}^n (2A_ixx^TA_i+ \varphi_i(x)A_i)\\
& =& \frac{2}{n}{\bf A}(x)^T{\bf A}(x) +\frac{1}{n}\sum_{i=1}^n  \varphi_i(x)A_i.
\eda
Several important probabilities are defined as follows.
\bea{lllllll}
p_1&:=&2\exp\left\lbrace -p\log(1+2n)\right\rbrace ,\quad
&p_2&:=&2\exp\left\lbrace -\frac{n \mathbb{E}\varepsilon_1^2}{512\varsigma    ^2} \right\rbrace ,\\
 p_3&:=&2\exp\left\lbrace-\frac{{\lambda}\|x^*\|^4n}{18{\eta}\varsigma    ^2}\right\rbrace,\quad
 &p_4&:=&2\exp\left\lbrace-\frac{n\gamma ^2\sigma ^4}{{ 4096\varrho^4}}\right\rbrace,
\eda
where the meaning of the parameters like $\varsigma,\lambda$ and $\gamma$ will be specified in the context.

\subsection{A blanket assumption and its implications}

The authors in  \cite{wang-xu19} introduced a phase retrieval property  of $n$ matrices $A_1,\ldots ,A_n$  in $\R^{p\times p}$. It was shown that model \eqref{PQMR} in the noiseless case (i.e., $\varepsilon_i=0$) and $n\geq2p-1$  has a unique solution up to a global sign.
To discuss the performance of element sample, we need the following  isometry property-like assumption, which is an extension of the frame defined in \cite{balan}.
\begin{assumption}\label{sample-RIP-general}
There exist two positive constants ${\lambda}$ and ${\eta}$ such that
\bea{ll}
{\lambda}\|u\| ^2\|v\| ^2\leq{}\frac{1}{n}\sum_{i=1}^n\langle u,A_iv \rangle ^2\leq{}{\eta}\|u\| ^2\|v\| ^2 \qquad \text{ for {all} }~u,v\in\mathbb{R}^p.
\nnb
\eda
\end{assumption}
{\noindent To show the existence of $A_i$ that satisfies the above assumption, we first introduce the definitions of sub-Gaussian variables and matrices \cite{RH15}. A random variable $X$ is said to be sub-Gaussian with variance proxy $\varrho^2$ if $\mathbb{E}X=0$ and its moment generating function satisfies $\mathbb{E}\exp(sX)\leq\exp(\varrho^2s^2/2)$ for any $s\in\mathbb{R}$.
A random matrix $X$ is said to be  sub-Gaussian with  proxy variance $\varrho^2$ if $\mathbb{E}X=0$ and $\langle u,Xv\rangle$ is sub-Gaussian with proxy variance $\varrho^2$ for any unit vectors $u$ and $v$. We use these sub-Gaussian matrices to construct $A_i$ as follows.}
\begin{setup}\label{ass-A-i}
Let $\{B_1,\ldots,B_n\}$ be a set of $p\times{}p$ sub-Gaussian random matrices with positive proxy $\varrho  ^2$ and assume that entries of each
$B_i$ are independent identical distributed with mean zero and positive variance $\sigma ^2$. Let $A_i=\big(B_i+B_i^T\big)/2$ for every
$i\in[n].$
\end{setup}
{ According to the definition of sub-Gaussian matrices, one can see that all entries of 
$B_i$ are sub-Gaussian with proxy variance no larger than $\varrho  ^2$  if $B_i$ is sub-Gaussian.} The above setup allows us to derive some useful properties as follows.
\begin{lemma}\label{sample-prop} If $\{A_1,\ldots,A_n\}$ are given as  Setup \ref{ass-A-i}, then the following statements are true.
\begin{itemize}[leftmargin=20pt]
\item[i)] Every $ A_i$ is a $p\times{}p$ symmetric sub-Gaussian random matrix with variance proxy $\varrho  ^2$. 
\item[ii)]  $\mathbb{E}(\langle u, A_iv\rangle^2)=\frac{1}{2}\sigma ^2(\|u\|^2\|v\|^2+\langle u,v\rangle^2)$ and hence
\bea{l}\frac{1}{2}\sigma ^2\|u\|^2\|v\|^2\leq\mathbb{E}(\langle u, A_iv\rangle^2)\leq\sigma ^2\|u\|^2\|v\|^2,\nnb
\eda
\item[iii)]  Every $\langle u, A_iv\rangle^2-\mathbb{E}\langle u, A_iv\rangle^2$ is a sub-exponential random variable with parameter $16\varrho  ^2\|u\|^2\|v\|^2$.
\end{itemize}
\end{lemma}
\begin{proof}
i) It is easy to check  any diagonal entry of $A_i, i=1,\ldots,n$ is sub-Gaussian random variable with proxy $\varrho  ^2$ and any off-diagonal entry is sub-Gaussian random variable with proxy $\varrho  ^2/2$. Denote the $(k,l)$th element of $A_i$ by $a_{kl}^i$. Direct calculation yields that \bea{ll}\label{uAv}
\langle u, A_iv\rangle=\sum_{i=1}^na_{kk}^iu_kv_k+\sum_{k=1}^{n-1}\sum_{l=k+1}^na_{kl}^i(u_kv_l+u_lv_k).
\eda
For any $s\in\mathbb{R}$ and { unit vectors $u, v\in\mathbb{R}^p$}, we conclude from the above equality and the independence of $a_{kl}^i$ for $k\leq{}l$ that
\bea{lll}
\mathbb{E}(\exp(s\langle u, A_iv\rangle))&=&\prod_{k=1}^p\mathbb{E}(\exp(sa_{kk}^iu_kv_k))
\prod_{k=1}^{p-1}\prod_{l=k+1}^{p}\mathbb{E}\big(\exp(sa_{kl}^i(u_kv_l+u_lv_k))\big)\\
&\leq&\prod_{k=1}^p\exp(s^2\varrho  ^2u_k^2v_k^2/2)\prod_{k=1}^{p-1}\prod_{l=k+1}^{p}\exp(s^2\varrho  ^2(u_kv_l+u_lv_k)^2/4)
\\
&\leq&\prod_{k=1}^p\exp(s^2\varrho  ^2u_k^2v_k^2/2)\prod_{k=1}^{p-1}\prod_{l=k+1}^{p}\exp(s^2\varrho  ^2((u_kv_l)^2+(u_lv_k)^2)/2)\\
&=&\exp(s^2\varrho  ^2\|u\|^2\|v\|^2/2),\nnb
\eda
where the first inequality derives from the fact that the diagonal entry and the off-diagonal entry of $A_i$ are the sub-Gaussian random variable with proxy $\varrho  ^2$ and  $\varrho  ^2/2$. This above condition  means $\langle u, A_iv\rangle$ is sub-Gaussian.

ii) We note that $\mathbb{E}(a_{kk}^i)^2=\sigma ^2$ and $\mathbb{E}(a_{kl}^i)^2=\sigma ^2/2$ for $k\neq{}l$.
By (\ref{uAv}), the independence of $a_{kl}^i$ for $k\leq{}l$ and $\mathbb{E}(A_i)=0$, we have
\bea{lll}\mathbb{E}(\langle u, A_iv\rangle)^2&=&\sum_{k=1}^n\mathbb{E}(a_{kk}^i)^2u_i^2v_i^2+\sum_{k=1}^{n-1}\sum_{l=k+1}^n\mathbb{E}(a_{kl}^i)^2(u_kv_l+u_lv_k)^2\\
&=&\sum_{k=1}^n\sigma ^2u_k^2v_k^2+\frac{1}{2}\sum_{k=1}^{n-1}\sum_{l=k+1}^n\sigma ^2(u_kv_l+u_lv_k)^2\\
&=&\frac{1}{2}\sigma ^2\|u\|^2\|v\|^2+\frac{1}{2}\sigma ^2(\sum_{i=1}^nu_k^2v_k^2+\sum_{k=1}^{n-1}\sum_{l=k+1}^n2u_kv_lu_lv_k)\\
&=&\frac{1}{2}\sigma ^2(\|u\|^2\|v\|^2+\langle u,v\rangle^2).
\nnb\eda
iii) The third result is a direct consequence of Lemma 1.12 of \cite{RH15}.\end{proof}

  Under Setup \ref{ass-A-i},  Assumption \ref{sample-RIP-general} can be guaranteed with a high probability, thereby displaying  the rationality of such an assumption, as outlined below.

{
\begin{lemma}\label{sample-RIP}  Let $\{A_1,\ldots,A_n\}$ be given as  Setup \ref{ass-A-i} and $\gamma\in(0,1/3)$.  If \bea{c}\label{np-RIP}n\geq\frac{4096\log(36\sqrt{2}/\gamma)\varrho^4}{\gamma^2\sigma^4}(2p+1),\eda then with a probability at least $1-p_4 $ there is
\bea{c}\label{subgauss-RIP}
\frac{1-2\gamma}{2}\sigma ^2\|u\|^2\|v\|^2\leq\frac{1}{n}\sum_{i=1}^n \left\langle u ,A_iv \right\rangle^2\leq\frac{2+3\gamma}{2}\sigma ^2\|u\|^2\|v\|^2~~~\mbox{for all}~u,v\in\R^p.
\eda
Furthermore, with the same probability,
\bea{c}\label{relation-3-ass}
\text{Setup \ref{ass-A-i}}~ \Longrightarrow ~~\text{Assumption \ref{sample-RIP-general}}.
\eda
\end{lemma}
\begin{proof} From Lemma \ref{sample-prop}  i) and iii), one conclude that $\left\langle u ,A_iv \right\rangle$ is a sub-Gaussian random variable with variance proxy $\varrho  ^2\|u\|^2\|v\|^2$ for any  nonzero vectors $u,v\in\mathbb{R}^p$ and  $\left\langle u ,A_iv \right\rangle^2-\mathbb{E}\left\langle u ,A_iv \right\rangle^2$ is sub-exponential random variable with parameter $16\varrho  ^2\|u\|^2\|v\|^2$. The independence of $\{A_1,\ldots,A_n\}$ and Bernstein's inequality enable us to obtain
\bea{lll}
~~~~\mathbb{P}\left(\left|\frac{1}{n}\sum_{i=1}^n \big(\left\langle u ,A_iv \right\rangle^2-\mathbb{E}\left\langle u ,A_iv \right\rangle^2\big)\right|>t\right) &\leq&
2\exp\left\lbrace  -\frac{nt}{32\|u\|^2\|v\|^2\varrho  ^2}\min\Big\{1,\frac{t}{16\|u\|^2\|v\|^2\varrho  ^2}\Big\}\right\rbrace
 \nnb
\eda
for any $t>0$. By choosing  $t=0.5\gamma\sigma ^2\|u\|^2\|v\|^2$, we conclude from the above  inequality, 
 Lemma \ref{sample-prop} ii), and $\sigma ^2\leq\varrho  ^2$ that
for any $u,v\in\R^p$,
\bea{c}
\mathbb{P}\Big(\Big\{\frac{1}{n}\sum_{i=1}^n \langle u ,A_iv\rangle^2<\frac{(1-\gamma)\sigma ^2\|u\|^2\|v\|^2}{2}\Big\}\bigcup\Big\{\frac{1}{n}\sum_{i=1}^n \langle u ,A_iv \rangle^2>\frac{(2+\gamma)\sigma ^2\|u\|^2\|v\|^2}{2}\Big\}\Big)\leq2\exp\left\lbrace-\frac{n\gamma^2 \sigma ^4}{2048\varrho  ^4}\right\rbrace.\nnb
\eda
 To ease of presentation, we denote $W:=vu^T$ and $\langle\cdot,\cdot\rangle$ the standard inner product of two square matrices. Moreover, we define $$S:=\{W\in\mathbb{R}^{p\times p}: \mathrm{rank}(W)=1, \|W\|_F=1\}.$$ Therefore, the above condition derives that for any $W\in S$,
\bea{c}\mathbb{P}\Big(\Big\{\frac{1}{n}\sum_{i=1}^n \left\langle A_i,W\right\rangle^2< \frac{1-\gamma}{2}\sigma^2\Big\}
\bigcup\Big\{\frac{1}{n}\sum_{i=1}^n \left\langle A_i,W\right\rangle^2> \frac{2+\gamma}{2}\sigma^2\Big\}\Big)\leq2\exp\left\lbrace-\frac{n\gamma^2 \sigma ^4}{2048\varrho  ^4}\right\rbrace,\nnb\eda
due to $\|W\|_F^2=\|u\|^2\|v\|^2$. We note that to prove (\ref{subgauss-RIP}), it suffices to prove 
\bea{c}\label{matrix-RIP2}\inf_{W\in S}\frac{1}{n}\sum_{i=1}^n \left\langle {A}_i,W\right\rangle^2\geq \frac{1-2\gamma}{2}\sigma^2
~\mbox{and}~ \sup_{W\in S}\frac{1}{n}\sum_{i=1}^n \left\langle {A}_i,W\right\rangle^2\leq \frac{2+3\gamma}{2}\sigma^2
\eda
with a probability at least $1-p_4 $.
We can prove the result using the similar reasoning to prove Theorem 2.3 in \cite{candes11}.  To keep the paper self-contained, we give the details. According to Lemma 3.1 in \cite{candes11}, there exists an $\epsilon$-net $\tilde{S}_\epsilon$ obeying
$$|\tilde{S}_\epsilon|\leq(9/\epsilon)^{2p+1}\qquad\text{and}\qquad \tilde{S}_\epsilon\supseteq S.$$  Here $|\tilde{S}_\epsilon|$ represents the cardinality of $\tilde{S}_\epsilon$. By taking $\epsilon=\gamma/(4\sqrt{2})$, we have
\bea{lll}&&\mathbb{P}\Big(\Big\{\inf_{\widetilde{W}\in \bar{S}_\epsilon}\frac{1}{n}\sum_{i=1}^n \langle A_i,\widetilde{W}\rangle^2< \frac{1-\gamma}{2}\sigma^2\Big\}
\bigcup\Big\{\sup_{\widetilde{W}\in \bar{S}_\epsilon}\frac{1}{n}\sum_{i=1}^n \langle A_i,\widetilde{W}\rangle^2> \frac{2+\gamma}{2}\sigma^2\Big\}\Big)\\
&\leq&\mathbb{P}\Big(\Big(\bigcup_{\widetilde{W}\in \bar{S}_\epsilon}\Big\{\frac{1}{n}\sum_{i=1}^n \langle A_i,\widetilde{W}\rangle^2< \frac{1-\gamma}{2}\sigma^2\Big\}\Big)
\bigcup\Big(\bigcup_{\widetilde{W}\in \bar{S}_\epsilon}\Big\{\frac{1}{n}\sum_{i=1}^n \langle A_i,\widetilde{W}\rangle^2> \frac{2+\gamma}{2}\sigma^2\Big\}\Big)\Big)\\
&\leq&2|\bar{S}_\epsilon|\exp\left\lbrace-\frac{n\gamma^2 \sigma ^4}{2048\varrho  ^4}\right\rbrace\\
&\leq&2\exp\left\lbrace-\frac{n\gamma^2 \sigma ^4}{2048\varrho  ^4}+(2p+1)\log(36\sqrt{2}/\gamma)\right\rbrace\\
&\leq&2\exp\left\lbrace-\frac{n\gamma^2 \sigma ^4}{4096\varrho  ^4}\right\rbrace=p_4. \qquad \text{(due to (\ref{np-RIP}))}
\nnb
\eda
 Therefore, with a probability at least $1-p_4 $, we have 
\bea{c}\inf_{\widetilde{W}\in \bar{S}_\epsilon}\frac{1}{n}\sum_{i=1}^n \langle A_i,\widetilde{W}\rangle^2\geq \frac{1-\gamma}{2}\sigma^2\quad\mbox{and}\quad
\sup_{\widetilde{W}\in \bar{S}_\epsilon}\frac{1}{n}\sum_{i=1}^n \langle A_i,\widetilde{W}\rangle^2\leq \frac{2+\gamma}{2}\sigma^2.\nnb
\eda
Let $\vartheta:=\sup_{W\in S}\sqrt{\frac{1}{n}\sum_{i=1}^n \left\langle {A}_i,W \right\rangle^2}.$ 
For any $W\in S$, there exists $\widetilde{W}\in\widetilde{S}_\epsilon$ satisfying $\|W-\widetilde{W}\|_F\leq\epsilon$. Using such $\widetilde{W}$ and  the triangle inequality, we can obtain 
\bea{lll}\label{kappa-ub}\sqrt{\frac{1}{n}\sum_{i=1}^n \langle {A}_i,W \rangle^2}&\leq&\sqrt{\frac{1}{n}\sum_{i=1}^n \langle {A}_i,\widetilde{W} \rangle^2}+\sqrt{\frac{1}{n}\sum_{i=1}^n \langle {A}_i,W-\widetilde{W} \rangle^2}\\
&\leq&\sigma\sqrt{1+\gamma/2}+\sqrt{\frac{1}{n}\sum_{i=1}^n \langle {A}_i,W-\widetilde{W}\rangle^2}.\eda
Since $\mathrm{rank}(W-\widetilde{W})\leq2$, one can use the singular value decomposition to get $W-\widetilde{W}=W_1+W_2$ with $\mathrm{rank}(W_1)=1,~ \mathrm{rank}(W_2)=1$, and
$\langle W_1,W_2\rangle=0$. This leads to
\bea{lll}\sqrt{\frac{1}{n}\sum_{i=1}^n \langle {A}_i,W-\widetilde{W} \rangle^2}&\leq&\sqrt{\frac{1}{n}\sum_{i=1}^n \left\langle {A}_i,W_1 \right\rangle^2}+\sqrt{\frac{1}{n}\sum_{i=1}^n \left\langle {A}_i,W_2 \right\rangle^2}\\
&\leq&\vartheta(\|W_1\|_F+\|W_2\|_F) \leq \sqrt{2}\vartheta \sqrt{\|W_1\|_F^2+\|W_2\|_F^2} \\
&=&  \sqrt{2}\vartheta \|W-\widetilde{W}\|_F \leq \sqrt{2}\vartheta\epsilon=\vartheta\gamma/4,\nnb\eda 
where the second inequality is due to $W_1/\|W_1\|_F,  W_2/\|W_2\|_F\in S$.
Plugging this into (\ref{kappa-ub}) gives
$$\mbox{$\sqrt{\frac{1}{n}\sum_{i=1}^n \left\langle {A}_i,W \right\rangle^2}$}\leq\sigma\sqrt{1+\gamma/2}+\vartheta\gamma/4.$$
Since this holds for all $W\in S$, it follows $\vartheta\leq\sigma\sqrt{1+\gamma/2}+\vartheta\gamma/4$, thereby 
$\vartheta\leq\sigma(1-\gamma/4)^{-1}\sqrt{1+\gamma/2}$.   One can check that $\sqrt{1+\gamma/2}\leq(1-\gamma/4)\sqrt{1+3\gamma/2}$ by $\gamma\in(0,1/3)$. Then $\vartheta\leq\sigma\sqrt{1+3\gamma/2}$, showing
 the upper bound in \eqref{matrix-RIP2}. Using again the triangle inequality,  we have
\bea{lll}\sqrt{\frac{1}{n}\sum_{i=1}^n \left\langle {A}_i,W \right\rangle^2} &\geq& \sqrt{\frac{1}{n}\sum_{i=1}^n  \langle {A}_i,\widetilde{W}  \rangle^2} - \sqrt{\frac{1}{n}\sum_{i=1}^n  \langle {A}_i,W-\widetilde{W}  \rangle^2} \\
&\geq&\sqrt{\frac{1-\gamma}{2}}\sigma- \frac{\vartheta\gamma}{4}\geq \sqrt{\frac{1-\gamma}{2}}\sigma-\frac{\gamma\sqrt{1+\gamma/2}}{4-\gamma}\sigma\\
&\geq&\sqrt{\frac{1-2\gamma}{2}}\sigma, \nnb\eda 
delivering the lower bound in \eqref{matrix-RIP2} as the above condition holds for all $W\in S$. Finally,  \eqref{relation-3-ass} can be ensured by setting ${\eta}=(2+3\gamma )\sigma^2/2$ and ${\lambda}=(1-2\gamma)\sigma^2/2$.
\end{proof}
\begin{remark}\label{remark-RIP} We emphasize that similar inequalities  to (\ref{subgauss-RIP}) can be guaranteed if $n\geq C p$ for a sufficiently large positive constant $C$, such as  \cite{machen2018}  for phase retrieval problems with real Gaussian samples  and \cite{Thaker} for QMR with complex Hermitian Gaussian random matrices.
\end{remark}
}

\subsection{Properties of $f$}
We  next study  the properties of $f$ under Assumption \ref{sample-RIP-general}. Those properties are beneficial to establish the convergence results for the algorithm proposed in the sequel.
\begin{lemma}\label{obj-coer1} If Assumption \ref{sample-RIP-general} holds, then $f(\cdot)$ is coercive, that is, $\lim_{\|x\|\to\infty} f( x )  =\infty.$
Hence, for any given $t\in\mathbb{R}$, the following level set of $f$ is compact,
\begin{eqnarray}\label{level set} L(t):=\{x\in\mathbb{R}^p: f(x)\leq{}t\}.\end{eqnarray}
\end{lemma}
\begin{proof}
By Cauchy's inequality and Assumption \ref{sample-RIP-general}, one can see that
\bea{lll}
f(x)=\frac{1}{4n}(\sum_{i=1}^n\langle x, A_ix\rangle^2-2\sum_{i=1}^n\langle x, A_ix\rangle b_i + \|b\|^2)
\geq\frac{{\lambda}}{4}\|x\|^4-\frac{\sqrt{{\eta}}}{2\sqrt{n}}\|x\|^2\|b\|+\frac{1}{4n}\|b\|^2,
\nnb\eda
which implies $f(x)\to\infty$ as $\|x\|\to\infty$. Because of this, $L(t)$ is bounded for any given $t$, which together with the continuity of $f(\cdot)$ yields the compactness.
\end{proof}
\noindent Assumption \ref{sample-RIP-general} enables us to bound the gradient and Hessian of $f$  on a given compact region.
 \begin{lemma}\label{fgh-bound}
 If Assumption \ref{sample-RIP-general} holds,
 then the following two terms
\bea{c}\label{bound-two-terms}
\sup_{\|x\|\leq t}\|\nabla{}f(x)\| \qquad \mbox{and}\qquad\sup_{\|x\|\leq t}\|\nabla^2f(x)\|
\eda
are bounded for any given  $t<+\infty$.
\end{lemma}
\begin{proof}
By (\ref{grad-hess}), Cauchy's inequality and Assumption \ref{sample-RIP-general}, one can see that
for any $u\in\mathbb{R}^p$,
\bea{lll}
|\langle u, \nabla{}f(x)\rangle| &\leq&\frac{1}{n}\sum_{i=1}^n |\varphi_i(x)||\langle u, A_ix\rangle|\\
&\leq&\sqrt{ \frac{1}{n}\sum_{i=1}^n\varphi^2_i(x)} \sqrt{\frac{1}{n}\sum_{i=1}^n\langle u, A_ix\rangle^2}  \leq  \sqrt{4{\eta}f(x)}\|u\|\|x\|,
\nnb\eda
which  leads to
\bea{lll}
\sup_{\|x\|\leq t} \|\nabla{}f(x)\| &=&\sup_{\|x\|\leq t} \sup_{\|u\|=1}|\langle u, \nabla{}f(x)\rangle|\\
& \leq& \sup_{\|x\|\leq t}  \sqrt{4{\eta}f(x)}\|x\| \leq   t\sqrt{4{\eta}\sup_{\|x\|\leq t} f(x)}<+\infty.
\nnb\eda
Similarly, one can check that
for any $u\in\mathbb{R}^p$,
\bea{lll}\langle u,\nabla^2f(x)u\rangle &\leq&\frac{2}{n}\sum_{i=1}^n \langle u, A_ix\rangle^2 +\frac{1}{n}\sum_{i=1}^n |\varphi_i(x)\langle u,A_iu\rangle| \\
&\leq&\frac{2}{n}\sum_{i=1}^n \langle u, A_ix\rangle^2 +\sqrt{\frac{1}{n}\sum_{i=1}^n \varphi^2_i(x)}
\sqrt{\frac{1}{n}\sum_{i=1}^n \langle u,A_iu\rangle^2} \\
&\leq&2{\eta}\|u\|^2\|x\|^2 +\sqrt{4{\eta}f(x)}\|u\|,
\nnb\eda
which gives rise to
\bea{lll}
\sup_{\|x\|\leq t} \|\nabla^2f(x)\|&=&\sup_{\|x\|\leq t} \sup_{\|u\|=1}|\langle u, \nabla^2f(x)u\rangle|
\leq\sup_{\|x\|\leq t}\left( 2{\eta}\|x\|^2 +\sqrt{4{\eta}f(x)}\right)\\  &\leq& 2{\eta}t^2 +\sqrt{4{\eta}\sup_{\|x\|\leq t} f(x)}<+\infty,
\nnb\eda
displaying the desired result.
\end{proof}
\noindent The next result reveals the Lipschitz continuity of the gradient and Hessian of $f$.
\begin{lemma}\label{hessian-lip} If Assumption \ref{sample-RIP-general} holds, then for any nonzero vectors $x,y\in\mathbb{R}^p$, it follows
\bea{cll}
\|\nabla^2f(x)-\nabla^2f(y)\|&\leq&3{\eta}\|x+y\|\|x-y\|,\\
\|\nabla f(x)-\nabla f(y)\|&\leq&\|\nabla^2f(y)\|\|x-y\|+2{\eta} \max\{\|x+y\|, \|x-y\|\}\|x-y\|^2 .
\nnb \eda
\end{lemma}
\begin{proof}
The Cauchy's inequality and Assumption \ref{sample-RIP-general} allow us to obtain for any $u\in\mathbb{R}^p$,
\bea{lll}&&|\langle u,(\nabla^2f(x)-\nabla^2f(y))u\rangle|\\
 &\leq&\frac{2}{n}\sum_{i=1}^n |\langle u,A_ix\rangle^2-\langle u, A_iy\rangle^2| +\frac{1}{n}\sum_{i=1}^n |(\langle x, A_ix\rangle-\langle y, A_iy\rangle)\langle u,A_iu\rangle|\\
 &=&\frac{2}{n}\sum_{i=1}^n |\langle u,A_i(x-y)\rangle\langle u, A_i(x+y)\rangle| +\frac{1}{n}\sum_{i=1}^n |(\langle x-y, A_i(x+y)\rangle\langle u,A_iu\rangle|\\
 &\leq&2\sqrt{\frac{1}{n}\sum_{i=1}^n \langle u,A_i(x-y)\rangle^2}\sqrt{\frac{1}{n}\sum_{i=1}^n \langle u, A_i(x+y)\rangle^2}\\
 & +&
 \sqrt{\frac{1}{n}\sum_{i=1}^n \langle x-y, A_i(x+y)\rangle^2}\sqrt{\frac{1}{n}\sum_{i=1}^n\langle u,A_iu\rangle^2}\\
&\leq&2{\eta}\|x+y\|\|x-y\|\|u\|^2+{\eta}\|x-y\|\|x+y\|\|u\|^2\\
&\leq&3{\eta}\|x+y\|\|x-y\|\|u\|^2,\nnb\eda
which together with $\|\nabla^2f(x)\|=\sup_{\|u\|=1}|\langle u, \nabla^2f(x)u\rangle|$ shows the first result. For the second one, by denoting $z_t:=y+t(x-y)$ for some $t\in(0,1)$ and using the first result, we obtain
\bea{llll}
\|\nabla^2f(z_t)-\nabla^2f(x)\|
=3{\eta}\|z_t+x\|\|z_t-x\|
\leq3{\eta}\left[\|x+y\|+t\|x-y\|\right](1-t)\|x-y\|
\nnb\eda
which enables us to derive the following chain of inequalities,
\bea{llll}
\|\nabla f(x)-\nabla f(y)\|
&=&\|\int_0^1\nabla^2f(z_t)(x-y)\mathrm{d}t\|\\
&\leq&\|\nabla^2f(x)(x-y)\|+\int_0^1\|\nabla^2f(z_t)-\nabla^2f(x)\|\|x-y\|\mathrm{d}t\\
&\leq&\|\nabla^2f(x)(x-y)\|+3{\eta} \|x-y\|^2 \int_0^1\left[\|x+y\|+t\|x-y\|](1-t)\right]\mathrm{d}t\\
&=&\|\nabla^2f(x)(x-y)\|+3{\eta} \|x-y\|^2  \left( \frac{1}{2}\|x+y\|+ \frac{1}{6}\|x-y\|\right) \\
&\leq&\|\nabla^2f(x)\|\|x-y\|+2{\eta} \|x-y\|^2  \max\{\|x+y\|, \|x-y\|\}.
\nnb\eda
The whole proof is completed.
\end{proof}
 {Hereafter, we define a local region $\mathcal{N}_*(r)$ around $x^*$  and some useful constants by
\bea{lll}\label{def-N-r*}
\mathcal{N}_*(r)&:=&\{x\in\mathbb{R}^p: \min\{\|x-x^*\|,\|x+x^*\|\}\leq r\},\\
\mathcal{N}_*&:=&\mathcal{N}_*(r_*),\qquad\text{where}\qquad r_*:= {{\lambda}}\|x^*\|/({6{\eta}}),\\
\epsilon_*&:=&\epsilon(r_*),~~\qquad\text{where}\qquad\epsilon(r):=(3{\lambda}-{\eta})r^3/32,\\
\kappa_*&:=&\|\frac{1}{n}\sum_{i=1}^n\varphi_i(x^*)A_i\|.\eda
Based on the above definitions, we have the following result. 
  \begin{lemma}\label{hessian-pd} Suppose Assumption \ref{sample-RIP-general} holds and assume $\kappa_* \leq ({{\lambda}}/{2})\|x^*\|^2$, then 
 \bea{l}
 \lambda_{\mathrm{min}}(\nabla^2f(x))\geq ({{\lambda}}/{3})\|x^*\|^2,~~\text{for all} ~x(\neq0)\in\mathcal{N}_*,\\
\langle\nabla f(x), x-x^*\rangle\geq({{\lambda}}/{3})\|x^*\|^2\|x-x^*\|^2-\kappa_*\|x^*\|\|x-x^*\|.\nnb\eda
\end{lemma}
\begin{proof}
We note from Assumption \ref{sample-RIP-general} that $\lambda_{\mathrm{min}}(\frac{2}{n}{\bf A}(x^*)^T{\bf A}(x^*))\geq2{\lambda}\|x^*\|^2$, which together with \eqref{grad-hess} and $\kappa_* \leq ({{\lambda}}/{2})\|x^*\|^2$ yields that
\bea{llll}\lambda_{\mathrm{min}}(\nabla^2f(x^*))\geq\lambda_{\mathrm{min}}(\frac{2}{n}{\bf A}(x^*)^T{\bf A}(x^*))-\|\frac{1}{n}\sum_{i=1}^n\varphi_i(x^*)A_i\|\geq({3{\lambda}}/{2})\|x^*\|^2.\nnb\eda
Without loss of the generality, we assume $\|x-x^*\|\leq\|x+x^*\|$. Thus for any $x\in\mathcal{N}_*$, we have $\|x-x^*\|\leq r_*$ and $\|x+x^*\|\leq \|x-x^*\|+2\|x^*\|\leq r_*+2\|x^*\|$. Using these conditions and Lemma \ref{hessian-lip}, for any $x\in\mathcal{N}_*$, we have
\bea{llll}
\lambda_{\mathrm{min}}(\nabla^2f(x))&\geq&\lambda_{\mathrm{min}}(\nabla^2f(x^*))-
\sup_{x\in\mathcal{N}_*}\|\nabla^2f(x)
-\nabla^2f(x^*)\|\\
&\geq&({3{\lambda}}/{2})\|x^*\|^2-3{\eta}\sup_{x\in\mathcal{N}_*}\big(
\|x+x^*\|\|x-x^*\|\big)\\
&\geq&({3{\lambda}}/{2})\|x^*\|^2-3{\eta}(r_*^2+2r_*\|x^*\|) \geq ({{\lambda}}/{3})\|x^*\|^2,\nnb\eda
where the last inequality is from  ${\lambda}\leq{\eta}$. We then get the first result by taking infimum over the above inequality. From \eqref{grad-hess}, we have 
\bea{llll}
\langle\nabla f(x^*), x-x^*\rangle\leq\|\nabla f(x^*)\|\|x-x^*\|\leq  \kappa_*\|x^*\|\|x-x^*\|.
\nnb\eda
For any $x\in\mathcal{N}_*$, it has $z^*_t:=x^*+t(x-x^*)\in\mathcal{N}_*$ for any $t\in[0,1]$. Then the mean value theorem allows us to derive that
\bea{llll}
\langle\nabla f(x), x-x^*\rangle&=&\langle\nabla f(x)-\nabla f(x^*), x-x^*\rangle+\langle\nabla f(x^*), x-x^*\rangle\\
&=&\langle\int_0^1\nabla^2f(z^*_t)(x-x^*)\mathrm{d}t, x-x^*\rangle+\langle\nabla f(x^*), x-x^*\rangle\\
&\geq&({{\lambda}}/{3})\|x^*\|^2\|x-x^*\|^2+\langle\nabla f(x^*), x-x^*\rangle\\
&\geq&({{\lambda}}/{3})\|x^*\|^2\|x-x^*\|^2-\kappa_*\|x^*\|\|x-x^*\|
\nnb\eda
displaying the desired result. 
\end{proof}
The following result is related to the strict saddle property.  According to \cite{Lee16, ge15}, a twice continuously differentiable function $f$ satisfies the strict saddle property if each stationary point $x$ of $f$ is either a local minimizer or a strict saddle (i.e., $\lambda_{\min}(\nabla^2 f(x))<0$).

 \begin{prop}\label{strict-saddle} Suppose Assumption \ref{sample-RIP-general} holds,  ${\eta}<3{\lambda}$, and $\kappa_*\leq {(3{\lambda}-{\eta})r^2}/{6}$ for a given $r>0$. Then for any point $x$, at least one of the following conditions holds,
\bea{c}\label{three-conditions}
{\rm c1)}~ x\in\mathcal{N}_*(r), ~~{\rm c2)}~  \|\nabla f(x)\|\geq\epsilon(r),~~{\rm c3)}~  \lambda_{\mathrm{min}}(\nabla^2f(x))\leq- {(3{\lambda}-{\eta})}\|x^*\|^2/8.
\eda 
In particular, we can take $r=r_*$.
\end{prop}
\begin{proof} It suffices to show that c3) holds if  $x\notin\mathcal{N}_*(r)$ and $\|\nabla f(x)\|<\epsilon(r)$. 
Let $x$ be any point such that $x\notin\mathcal{N}_*(r)$ and $\|\nabla f(x)\|<\epsilon(r)$. For notational simplicity, denote $x^-:=x-x^*$ and $x^+:=x+x^*$. Without loss of generality, we let $\|x^-\|\leq\| x^+\|$ (The following analysis will be centred on $x^+$ if $\|x^-\|\geq\| x^+\|$). Then 
\bea{lll}\label{xplus-lb}
\|x^*\|=\frac{1}{2}\|x^*+x-x+x^*\|\leq\frac{1}{2}\|x^-\|+\frac{1}{2}\|x^+\|\leq \|x^+\|.\eda 
By $4x-x^-=x^-+2x^+$ and $\varphi_i(x)-\varphi_i(x^*)=\langle x^{-},A_ix^+\rangle$, we have
\bea{lll}
&&\langle x^-, \nabla^2f(x) x^-  \rangle\\
&=&\frac{1}{2n}\sum_{i=1}^n \left[4\langle x^-, A_ix\rangle^2+  2\varphi_i(x)\langle x^-, A_ix^-\rangle \right]\nonumber\\
&=&\frac{1}{2n}\sum_{i=1}^n\left[  \langle x^-, 2A_ix\rangle^2+ 2\varphi_i(x)\langle x^-, A_i(4x-4x+x^-)\rangle\right]\nonumber\\
&=&\frac{1}{2n}\sum_{i=1}^n\left[ \langle x^-, A_i(x^-+x^+)\rangle^2+2 \varphi_i(x)\langle x^-, A_i(x^--4x)\rangle
+8 \varphi_i(x)\langle x^-, A_ix\rangle\right]\nonumber\\
&=&\frac{1}{2n}\sum_{i=1}^n\left[ \langle x^-, A_i(x^-+x^+)\rangle^2-2 \varphi_i(x)\langle x^-, A_i(x^-+2x^+)\rangle
\right]+4\langle\nabla f(x),x^-\rangle\nonumber\\
&=&\frac{1}{2n}\sum_{i=1}^n\big[ \langle x^-, A_i(x^-+x^+)\rangle^2-2 (\varphi_i(x)-\varphi_i(x^*))\langle x^-, A_i(x^-+2x^+)\rangle\\
&&\qquad\qquad+\varphi_i(x^*)\langle x^-, A_i(x^-+2x^+)\rangle\big]+4\langle\nabla f(x),x^-\rangle\nonumber\\
&=&\frac{1}{2n}\sum_{i=1}^n \big[\langle x^-, A_i(x^-+x^+)\rangle^2 - (\langle x^-, A_ix^+\rangle)(2\langle x^-, A_ix^-\rangle+4 \langle x^-, A_ix^+\rangle)\\
&&\qquad\qquad+\varphi_i(x^*)\langle x^-, A_i(x^-+2x^+)\rangle\big]+4\langle\nabla f(x),x^-\rangle\nonumber\\
&=&\frac{1}{2n}\sum_{i=1}^n \big[(\langle x^-, A_ix^-\rangle +\langle x^-, A_ix^+\rangle)^2-2\langle x^-, A_ix^+\rangle\langle x^-, A_ix^-\rangle-4\langle x^-, A_ix^+\rangle^2\\
&&\qquad\qquad+\varphi_i(x^*)\langle x^-, A_i(x^-+2x^+)\rangle\big]+4\langle\nabla f(x),x^-\rangle\nonumber\\
&=&\frac{1}{2n}\sum_{i=1}^n  \left[\langle x^-, A_ix^-\rangle  ^2 - 3   \langle x^-, A_ix^+\rangle^2+\varphi_i(x^*)\langle x^-, A_i(x^-+2x^+)\rangle\right]+4\langle\nabla f(x),x^-\rangle\nonumber.
\eda
Now we conclude form the above condition and Assumption \ref{sample-RIP-general}   that
\bea{lll}
&&2\langle x^-, \nabla^2f( x) x^-  \rangle\\
&\leq&
 {\eta}\|x^-\|^4-3{\lambda}\|x^-\|^2\|x^+\|^2+|\frac{1}{n}\sum_{i=1}^n\varphi_i(x^*)\langle x^-, A_i(x^-+2x^+)\rangle|+
 8\|\nabla f(x)\|\|x^-\|\\
 &\leq&
 {\eta}\|x^-\|^4-3{\lambda}\|x^-\|^2\|x^+\|^2+\|\frac{1}{n}\sum_{i=1}^n\varphi_i(x^*)A_i\|\|x^-\|\|x^-+2x^+\|+
 8\|\nabla f(x)\|\|x^-\|\\
 &\leq&
-(3{\lambda}-{\eta})\|x^-\|^2\|x^+\|^2 +3\kappa_*\|x^-\|\|x^+\|+8\|\nabla f(x)\|\|x^-\|\\
&=&\Big(-(3{\lambda}-{\eta})+\frac{3\kappa_*}{\|x^-\|\|x^+\|}+\frac{8\|\nabla f(x)\|}{\|x^-\|\|x^+\|^2}\Big)\|x^-\|^2\|x^+\|^2\\
&\leq&\Big(-(3{\lambda}-{\eta})+3\kappa_* r^{-2}+8\|\nabla f(x)\|r^{-3}\Big)\|x^-\|^2\|x^+\|^2\qquad(\text{by $x\notin\mathcal{N}_*(r)$})\\
&\leq&-\frac{3{\lambda}-{\eta}}{4}\|x^-\|^2\|x^+\|^2 \leq -\frac{3{\lambda}-{\eta}}{4}\|x^-\|^2\|x^*\|^2,
\nnb\eda
 where  the fifth inequality follows  from conditions  ${\eta}<3{\lambda}$, $\kappa_*\leq({3{\lambda}-{\eta}})r^2/6$, and
 $\|\nabla f(x)\|\leq \epsilon(r)=(3{\lambda}-{\eta})r^3/32$, and the last inequality is from (\ref{xplus-lb}). The above condition   shows c3).
\end{proof}
\begin{remark}\label{remark-ssp}
In the noiseless case, $\kappa_*=0$ due to $\varphi_i(x^*)=0$ for all $i\in[n]$, thereby condition $\kappa_*\leq({3{\lambda}-{\eta}})r^2/6$ holds evidently. In the noisy case, Lemma \ref{err-sup-ub} in next section states that $\kappa_*$ is smaller than a threshold in a high probability.   Therefore, it is rational to assume $\kappa_*\leq({3{\lambda}-{\eta}})r^2/6$.
 
 Moreover,  based on Proposition \ref{strict-saddle} with taking $r=r_*$,   for any point whose gradient is too small to meet condition c2) in (\ref{three-conditions}), it is either in $\mathcal{N}_*=\mathcal{N}_*(r_*)$ or close to a strict saddle point. While from Lemma \ref{hessian-pd}, any points in $\mathcal{N}_*$ can not be a strict saddle point because the first condition in Lemma \ref{hessian-pd}. Therefore, we can conclude that any stationary point $w\in\mathcal{N}_*$ (i.e., $\nabla f(w)=0$) is a unique local minimizer to  (\ref{LSE-PQMR}) due to $w\in{\rm argmin}_{x\in\mathcal{N}_*} f(x)$ and the strong convexity of $f$ on $\mathcal{N}_*$.

Finally, for the noiseless case, there is $\kappa_*=0$. The above statement means that any stationary point $w\in\mathcal{N}_*$  satisfies $w={\rm argmin}_{x\in\mathcal{N}_*} f(x)$. Therefore, $0\leq f(w)={\rm min}_{x\in\mathcal{N}_*} f(x)\leq f(x^*)=0$, which leads to $w=x^*$ or $w=-x^*$. In other words, for the noiseless case if we find a stationary point in $\mathcal{N}_*$, then this stationary point is the true signal.
\end{remark}}
\section{Non-asymptotic Statistical Analysis}\label{sec:sta-ana}
In this section, we would like to see the reconstruction performance of model \eqref{LSE-PQMR} for original problem \eqref{PQMR}. That is, we aim at establishing the error bound of the gap between recovered signal $\hat x$ and true signal $x^*$.  To proceed with that, we need some assumptions on noise $\varepsilon_i$.

\begin{assumption}\label{ass-error}
Let errors $\{\varepsilon_1,\ldots,\varepsilon_n\}$ be i.i.d. sub-Gaussian with
positive proxy $\varsigma^2$ and $\varepsilon_1$ have mean zero and positive variance $\mathbb{E}\varepsilon_1^2$.
\end{assumption}

{
\begin{lemma}\label{err-sup-ub}
If  Assumptions \ref{sample-RIP-general} and \ref{ass-error} hold, then
\bea{lll}\mathbb{P}\Big( \sup_{u,v\in\B} |\frac{1}{n}\sum_{i=1}^n\left\langle u ,A_iv \right\rangle\varepsilon_i |>
2\sqrt{\frac{6\varsigma    ^2{\eta}p\log(1+2n)}{n}}\Big)\leq p_1+p_2\nnb\eda
and
\bea{lll}\mathbb{P}\Big( \kappa_*>
2\sqrt{\frac{6\varsigma    ^2{\eta}p\log(1+2n)}{n}}\Big)\leq p_1+p_2.\nnb\eda
\end{lemma}
\begin{proof}
Assumption \ref{ass-error} and \cite[Lemma 1.12]{RH15} imply that  $\varepsilon_i^2-\mathbb{E}\varepsilon_i^2$ is sub-exponential with parameter $16\varsigma    ^2$. From Bernstein's inequality, for any $t>0$,
\bea{c}
\mathbb{P}\left(\frac{1}{n}\left|\sum_{i=1}^n(\varepsilon_i^2-\mathbb{E}\varepsilon_i^2)\right|>t\right)\leq2 \exp\left\lbrace -\frac{nt}{32\varsigma    ^2}\min\left\{\frac{t}{16\varsigma    ^{2}},1\right\}\right\rbrace ,
\nnb\eda
which together with the choice of $t=\mathbb{E}\varepsilon_1^2 \leq \varsigma ^{2}$ due to the sub-Gaussian yields that
\bea{c}
\label{ei-bound}\mathbb{P}\left(\frac{1}{n}\sum_{i=1}^n\varepsilon_i^2>{}2\mathbb{E}\varepsilon_1^2\right)\leq2\exp\left\lbrace-\frac{n\mathbb{E}\varepsilon_1^2 }{512\varsigma    ^2} \right\rbrace = p_2.
\eda For any nonzero vectors $u, v\in\mathbb{R}^p$ and $t>0$, it follows from \cite[Corollary 1.7]{RH15}   that
\bea{c}
\mathbb{P}\left(\left|\frac{1}{n}\sum_{i=1}^n \left\langle u ,A_iv \right\rangle\varepsilon_i\right|>t\right)\leq
2\exp\left\lbrace -\frac{n^2t^2}{2\varsigma    ^2\sum_{i=1}^n \left\langle u ,A_iv \right\rangle^2}\right\rbrace ,\nnb
\eda
 which together with Assumption \ref{sample-RIP-general} results in
\bea{c}\label{uav-error}
\mathbb{P}\left(\left|\frac{1}{n}\sum_{i=1}^n  \left\langle u ,A_iv \right\rangle\varepsilon_i\right|>t\right)\leq
2\exp\left\lbrace -\frac{nt^2}{2\varsigma    ^2{\eta}\|u\|^2\|v\|^2}\right\rbrace .
\eda
For any $u\in\B$, let $$ \B(u,\delta):=\{x\in\B:\|x-u\|\leq{}\delta\}.$$ It follows from \cite[Lemma 14.27]{bugeer11} that there exist $J:=(1+2n)^p$ vectors $u_j\in\B$ such that
\bea{l}\label{def-B-cover}
\B\subseteq\bigcup_{j=1}^{J}\B\left( u_j,\frac{1}{n}\right) .
\eda
Using this fact,  for any fixed $u\in\B$, we obtain the following chain of inequalities
\bea{lll}
&&\sup_{u\in\B}|\frac{1}{n}\sum_{i=1}^n \left\langle u,A_iv\right\rangle\varepsilon_i| \\
&\leq&\sup_{u\in\bigcup_{j=1}^{J  }\B\left( u_j,\frac{1}{n}\right)}|\frac{1}{n}\sum_{i=1}^n \left\langle u,A_iv\right\rangle\varepsilon_i| \\
&=&\sup_{j\in[J  ]}\Big(|\frac{1}{n}\sum_{i=1}^n \left\langle u_j,A_iv\right\rangle\varepsilon_i|
+\sup_{u\in{}\B\left( u_j,\frac{1}{n}\right)}|\frac{1}{n}\sum_{i=1}^n \left\langle u-u_j,A_iv\right\rangle \varepsilon_i|\Big) \\
&\leq&\sup_{j\in[J  ]}\Big(|\frac{1}{n}\sum_{i=1}^n \left\langle u_j,A_iv\right\rangle\varepsilon_i|
+\sup_{u\in{}\B\left( u_j,\frac{1}{n}\right)}\sqrt{\frac{1}{n}\sum_{i=1}^n \left\langle u-u_j,A_iv\right\rangle^2}\sqrt{\frac{1}{n}\sum_{i=1}^n \varepsilon_i^2}\Big) \\
&\leq&\sup_{j\in[J  ]}\Big(|\frac{1}{n}\sum_{i=1}^n \left\langle u_j,A_iv\right\rangle\varepsilon_i|+\sup_{u\in{}\B\left( u_j,\frac{1}{n}\right)}\sqrt{\frac{{\eta}}{n}}\|v\|\|u-u_j\|\|\varepsilon\|\Big) \\
&\leq&\sup_{j\in[J  ]}|\frac{1}{n}\sum_{i=1}^n \left\langle u_j,A_iv\right\rangle\varepsilon_i|
+\frac{1}{n} \frac{\sqrt{{\eta}}}{\sqrt{n}} \|\varepsilon\|.\nnb
\eda
There are also $J$ vectors $v_k\in\B$ such that \eqref{def-B-cover}, namely,
$\B\subseteq\bigcup_{k=1}^{J}\B\left(v_k,\frac{1}{n}\right).$ Same reasoning allows us to obtain, for each fixed $u_j$,
\bea{lll}
\sup_{v\in\B}|\frac{1}{n}\sum_{i=1}^n \left\langle u_j,A_iv\right\rangle \varepsilon_i|
\leq\sup_{k\in[J  ]}|\frac{1}{n}\sum_{i=1}^n \left\langle u_j,A_iv_k\right\rangle\varepsilon_i|
+\frac{1}{n} \frac{\sqrt{{\eta}}}{\sqrt{n}} \|\varepsilon\|. \nonumber
\eda
 Using the above two facts delivers
\bea{lll}\label{norm2-number}
\sup_{u,v\in\B} |\frac{1}{n}\sum_{i=1}^n \left\langle u ,A_iv \right\rangle \varepsilon_i |
&\leq&
\sup_{v\in\B}\left(\sup_{j\in[J  ]}|\frac{1}{n}\sum_{i=1}^n \left\langle u_j,A_iv \right\rangle\varepsilon_i|
+\frac{1}{n} \frac{\sqrt{{\eta}}}{\sqrt{n}} \|\varepsilon\|\right) \\
&=&\sup_{j\in[J  ]}\sup_{v\in\B}|\frac{1}{n}\sum_{i=1}^n \left\langle u_j,A_iv \right\rangle\varepsilon_i|
+\frac{1}{n} \frac{\sqrt{{\eta}}}{\sqrt{n}} \|\varepsilon\|\\
&\leq&\sup_{j,k\in[J  ]}|\frac{1}{n}\sum_{i=1}^n \left\langle u_j,A_iv_k\right\rangle\varepsilon_i|
+\frac{1}{n} \frac{\sqrt{{\eta}}}{\sqrt{n}} \|\varepsilon\|.
\eda
Now direct calculation enables us to derive the following chain of inequalities
\bea{llll}
&&\mathbb{P}\Big( \sup_{u,v\in\B} |\frac{1}{n}\sum_{i=1}^n\left\langle u ,A_iv \right\rangle\varepsilon_i |>
t+\frac{\sqrt{2{\eta}\mathbb{E}\varepsilon_1^2}}{n}\Big)\\
&\leq&\mathbb{P}\Big( \sup_{j,k\in[J  ]} \left|\frac{1}{n}\sum_{i=1}^n \left\langle u_j,A_iv_k\right\rangle\varepsilon_i\right|
+\frac{1}{n} \frac{\sqrt{{\eta}}}{\sqrt{n}} \|\varepsilon\|>t+\frac{\sqrt{2{\eta}\mathbb{E}\varepsilon_1^2}}{n}\Big)\qquad&(\text{by \eqref{norm2-number}})\\
&\leq&\mathbb{P}\left( \frac{1}{n}\|\varepsilon\|^2>2\mathbb{E}\varepsilon_1^2\right)+\mathbb{P}\left( \sup_{j,k\in[J  ]} \left|\frac{1}{n}\sum_{i=1}^n \left\langle u_j,A_iv_k\right\rangle\varepsilon_i\right|>
t\right)
\\
&\leq&\mathbb{P}\left( \frac{1}{n}\|\varepsilon\|^2>2\mathbb{E}\varepsilon_1^2\right)+\sum_{j,k\in[J  ]}\mathbb{P}\left( \left|\frac{1}{n}\sum_{i=1}^n \left\langle u_j,A_iv_k\right\rangle\varepsilon_i\right|>
t\right)
\\
&\leq&p_2+\sum_{j,k\in[J  ]}\exp\left\lbrace -\frac{nt^2}{2\varsigma    ^2{\eta}\|u_j\|^2\|v_k\|^2}\right\rbrace \qquad &(\text{by (\ref{ei-bound})  and (\ref{uav-error})})\\
&\leq&p_2+2J  ^2 \exp\left\lbrace -\frac{nt^2}{2\varsigma    ^2{\eta}}\right\rbrace
\\
&=&p_2+2\exp\left\lbrace -\frac{nt^2}{2\varsigma    ^2{\eta}}+2p\log(1+2n)\right\rbrace
,\nnb
\eda
where the last inequality is due to $u_j,v_k\in\B$. By taking
\bea{l}t=\sqrt{\frac{6\varsigma    ^2{\eta}p\log(1+2n)}{n}},\nnb\eda
we derive the probability
\bea{lll}
2\exp\left\lbrace -\frac{nt^2}{2\varsigma    ^2{\eta}}+2p\log(1+2n) \right\rbrace  &= &2\exp\left\lbrace-p\log(1+2n)\right\rbrace
=p_1,\nnb\eda
and thus can claim that
\bea{lll}\sup_{u,v\in\B}\Big|\frac{1}{n}\sum_{i=1}^n\left\langle u ,A_iv \right\rangle\varepsilon_i\Big|&>&
\sqrt{\frac{6\varsigma    ^2{\eta}p\log(1+2n)}{n}}+\frac{\sqrt{2{\eta}\mathbb{E}\varepsilon_1^2}}{n},\nnb
\eda
holds with a probability no more than $p_1+p_2$. Since $np\log(1+2n) \geq {\mathbb{E}\varepsilon_1^2}/({3\varsigma    ^2})$ due to the property of sub-Gaussian random variable $\mathbb{E}\varepsilon_1^2\leq\varsigma    ^2$, it follows \bea{lll}\frac{\sqrt{2{\eta}\mathbb{E}\varepsilon_1^2}}{n}
\leq\sqrt{\frac{6\varsigma^2{\eta}p\log(1+2n)}{n}}\nonumber\eda
 which after simple manipulations derives the first result. As a direct result, the second follows from the fact that $\varphi_i(x^*)=-\varepsilon_i$ and $\|A\|=\sup_{u\in\mathbb{B}}\langle u, Au\rangle$ for a symmetric matrix $A$.
\end{proof}
}

\begin{lemma}\label{non-empty-level-set}
If   Assumption \ref{sample-RIP-general} holds, then $f(x^*)<f(0)$ for the noiseless case if  $\|b\|>0$ and \bea{l}\mathbb{P}(f(x^*)<f(0))\geq 1- p_3\nnb\eda for the noisy case under Assumption \ref{ass-error}.
\end{lemma}
\begin{proof} For the noiseless case,  the claim is clearly true due to
\bea{c}f(x^*)=0<\frac{1}{4n}\|b\|^2=f(0).\nnb
\eda
For the noisy case, the inequality (\ref{uav-error}) implies that
\bea{c}
\mathbb{P}\left(\left|\frac{1}{n}\sum_{i=1}^n \langle x^*, A_ix^*\rangle \varepsilon_i\right|>\frac{{\lambda}}{3}\|x^*\|^4\right)
\leq2\exp\left\lbrace-\frac{n{\lambda}\|x^*\|^4}{18{\eta}\varsigma    ^2}\right\rbrace =p_3.
\nnb\eda
Using the above fact and Assumption \ref{sample-RIP-general},  we conclude that
\bea{lll}
f(x^*)-f(0)&=&\frac{1}{4n}\|\varepsilon\|^2-\frac{1}{4n}\| b\|^2\\
&=&\frac{1}{4n}\|\varepsilon\|^2-\frac{1}{4n} \sum_{i=1}^n( \langle x^*, A_ix^*\rangle + \varepsilon_i) ^2\qquad(\text{by \eqref{PQMR}})\\
&=&-\frac{1}{4n}\sum_{i=1}^n \langle x^*, A_ix^*\rangle ^2-\frac{1}{2n}\sum_{i=1}^n \langle x^*, A_ix^*\rangle \varepsilon_i\\
&\leq&-\frac{{\lambda}}{4}\|x^*\|^4+\frac{{\lambda}}{6}\|x^*\|^4\\
&=&-\frac{{\lambda}}{12}\|x^*\|^4<0
\nnb\eda
with a probability at least $1-p_3$.
\end{proof}

The above lemma indicates that $L(f(0))=\{ {x}\in\mathbb{R}^p: f( {x})<f( {0})\}$ is non-empty with a high probability and hence model \eqref{LSE-PQMR} can exclude a zero solution. In fact, if  level set $L(f(0))$ is empty, then $0$ is a stationary point of \eqref{LSE-PQMR} since $\nabla f(0)=0$ and thus is also a globally optimal solution to \eqref{LSE-PQMR}. However, $x^*=0$ is trivial and is not what we pursue. Therefore,
throughout the paper, we always assume $L(f(0))$ is non-empty if no additional explanations are provided.

Now we are ready to present our main result. It establishes a non-asymptotic inequality revealing that global minimum $\hat x$ of model \eqref{LSE-PQMR} can capture true signal $x^*$ exactly for the noiseless case while achieve an error rate $O(\sqrt{p\log(1+2n)/{n}})$ with a high probability for the noisy case. In other words, the global minimum of \eqref{LSE-PQMR} is unique for the noiseless case.

\begin{theorem}\label{nonasym-bound}
If   Assumption \ref{sample-RIP-general} holds, then $ \hat x=x^*$ or $ \hat x=-x^*$ for the noiseless case and
\bea{c}\label{hatx-asymp}
\min\{\|\hat{x}-x^*\|,\|\hat{x}+x^*\|\}\leq\frac{4}{{\lambda}\|x^*\|}
\sqrt{\frac{6\varsigma    ^2{\eta}p\log(1+2n)}{n}}
\eda
 with a probability at least $1-p_1-p_2$ for the noisy case under Assumption \ref{ass-error}.
\end{theorem}
\begin{proof}For the noiseless case, from $0=f(x^*)\geq f(\hat{x})\geq 0$, we have $f(x^*) = f(\hat{x})=0$. Then by $b_i=\langle x ^*, A_ix^*\rangle$ and Assumption \ref{sample-RIP-general}, we can conclude that
\bea{lll}
0&=&f(\hat{x}) - f(x^*)\\
&=& \frac{1}{4n}\sum_{i=1}^n (\langle\hat{x}, A_i\hat{x}\rangle-b_i)^2-\frac{1}{n}\sum_{i=1}^n(\langle x ^*, A_ix^*\rangle-b_i)^2\\
&=& \frac{1}{4n}\sum_{i=1}^n (\langle\hat{x}, A_i\hat{x}\rangle-\langle x ^*, A_ix^*\rangle)^2\\
&=&\frac{1}{4n}\sum_{i=1}^n \langle\hat{x}-x^*, A_i(\hat{x}+x^*)\rangle^2\\
&\geq&\frac{{\lambda}}{4}\|\hat{x}-x^*\|^2\|\hat{x}+x^*\|^2,\nnb
\eda
which indicates $\|\hat{x}-x^*\|\|\hat{x}+x^*\|=0$, showing the conclusion.

For the noisy case, if $\|\hat{x}-x^*\|\|\hat{x}+x^*\|=0$, then the conclusion is true clearly. So, we only focus on $\|\hat{x}-x^*\|\|\hat{x}+x^*\|\neq0$.
From $f(x^*)\geq f(\hat{x})$, $b_i=\langle x ^*, A_ix^*\rangle+\varepsilon_i$, and Assumption \ref{sample-RIP-general}, we conclude that
\bea{lll}\label{fun-xtrue-bound}
0
&\geq&
 \frac{1}{4n}\sum_{i=1}^n (\langle\hat{x}, A_i\hat{x}\rangle-b_i)^2-\frac{1}{4n}\sum_{i=1}^n(\langle x ^*, A_ix^*\rangle-b_i)^2\\
&=& \frac{1}{4n}\sum_{i=1}^n (\langle\hat{x}, A_i\hat{x}\rangle-\langle x ^*, A_ix^*\rangle-\varepsilon_i)^2-\frac{1}{4n}\sum_{i=1}^n\varepsilon_i^2\\
&=&\frac{1}{4n}\sum_{i=1}^n (\langle\hat{x}, A_i\hat{x}\rangle-\langle x ^*, A_ix^*\rangle)^2-\frac{1}{2n}\sum_{i=1}^n (\langle\hat{x}, A_i\hat{x}\rangle-\langle x ^*, A_ix^*\rangle)\varepsilon_i\\
&=&\frac{1}{4n}\sum_{i=1}^n \langle\hat{x}-x^*, A_i(\hat{x}+x^*)\rangle^2
-\frac{1}{2n}\sum_{i=1}^n \langle\hat{x}-x^*, A_i(\hat{x}+x^*)\rangle\varepsilon_i\\
&\geq&\frac{{\lambda}}{4}\|\hat{x}-x^*\|^2\|\hat{x}+x^*\|^2
-\frac{1}{2n}\sum_{i=1}^n \langle\hat{x}-x^*, A_i(\hat{x}+x^*)\rangle\varepsilon_i,
\eda
which by $\frac{\hat{x}-x^*}{\|\hat{x}-x^*\|},\frac{\hat{x}+x^*}{\|\hat{x}+x^*\|}\in\B$ results in
\bea{lll} \label{est-bound-00}
 \|\hat{x}-x^*\|\|\hat{x}+x^*\|
\leq \sup_{u,v\in\B}  \Big|\frac{2}{n{\lambda}}
\sum_{i=1}^n \left\langle u,A_iv\right\rangle \varepsilon_i \Big|.
\eda
If $\|\hat{x}+x^*\|\geq \|\hat{x}-x^*\|$, then $\langle \hat{x}, x^* \rangle \geq 0$ and hence $\|\hat{x}+x^*\| \geq \|x^*\|$. Similarly, if $\|\hat{x}+x^*\|\leq \|\hat{x}-x^*\|$, then $\langle \hat{x}, x^* \rangle \leq 0$ and hence $\|\hat{x}-x^*\| \geq \|x^*\|$. Overall,
\bea{lll}
\|\hat{x}-x^*\|\|\hat{x}+x^*\|\geq\|x^*\|\min\{\|\hat{x}-x^*\|,\|\hat{x}+x^*\|\}.
\nnb\eda
This together with \eqref{est-bound-00}  yields that
\bea{lll}\label{est-bound-1}
\min\{\|\hat{x}-x^*\|,\|\hat{x}+x^*\|\}
&\leq&\sup_{u,v\in\B} \left|\frac{2}{n{\lambda}\|x^*\|}
\sum_{i=1}^n \left\langle u,A_iv\right\rangle\varepsilon_i\right|.\nonumber
\eda
Combing this and Lemma \ref{err-sup-ub}, we have
\bea{c}
\min\{\|\hat{x}-x^*\|,\|\hat{x}+x^*\|\}\leq\frac{4}{{\lambda}\|x^*\|}\left(
\sqrt{\frac{6\varsigma    ^2{\eta}p\log(1+2n)}{n}}\right)\nonumber
\eda
 with  a probability at least $1-p_1-p_2$.
\end{proof}
\noindent A direct result of the above theorem is the consistency of $\hat{x}$.
\begin{corollary}\label{consistency}Suppose Assumptions \ref{sample-RIP-general} and \ref{ass-error} hold. If $p\log(1+2n)/n\to0$ as $n\to\infty$, then
\bea{c}\min\{\|\hat{x}-x^*\|,\|\hat{x}+x^*\|\}\to0.\nnb\eda \end{corollary}

\section{Gradient Regularized Newton Method}\label{sec:alg}
In the previous section, we have shown that the constructed  signal from \eqref{LSE-PQMR} can either capture the true signal  exactly or achieve an error bound with a high probability under some assumptions. Therefore, it is rational to pay attention to pursue a solution to \eqref{LSE-PQMR}. In this section, we aim to develop a numerical algorithm to solve it, before which, for notational simplicity, we denote
\bea{c}\label{def-H-k}
g^k:=\nabla{}f(x^{k}),\qquad H^k:=\frac{2}{n}{\bf A}(x^{k})^T{\bf A}(x^{k}).
\nnb\eda

\subsection{Algorithmic framework}
{Proposition \ref{strict-saddle}
states that the whole space, $\mathbb{R}^p$, can be divided into three
regions: 1) the region close to a local minimum; 2) a region where the gradient is large, and 3) a region around saddle point whose Hessian has a significant
negative eigenvalue. 
Intuitively, when the gradient is large, the function value
will decrease quickly  by gradient descent; when the point
is close to a saddle point, an inexact gradient could help the algorithm avoid the saddle
point and the function value will also decrease; when the point
is close to a local minimum point, the algorithm equipped with Newton directions can converge 
to a local minimum with a faster speed than the gradient descent. However, it is well known that the Newton method does not treat saddle points appropriately.}
Inspired by the regularized Newton method \cite{ueda10}, we present a modified  version in Algorithm \ref{algorithm 0}.

\begin{algorithm}[!th]
\SetAlgoLined
Initialize a bounded $x^0\in \mathbb{R}^{p}$, $ \epsilon_G, K_0, \beta>0$, $\mu_{1} , \mu_{2} , \alpha_1,\alpha_2,\delta\in(0,1)$.    Set $ k\Leftarrow0$.

\emph{(Phase I $-$ gradient descent)}

\While{$\|g^k\| \geq\epsilon_G $ {or $k\leq K_0$}}{
Find the smallest nonnegative integer $j_k$  such that
\begin{eqnarray}
\label{armijo-sg}
f(x^k-\alpha_1^{j_k}g^k)\leq{}f(x^k)-
\mu_{1} \alpha_1^{j_k}\|g^k \| ^2.\end{eqnarray}

Set $\tau_k=\alpha_1^{j_k}, x^{k+1}=x^k-\tau_kg^k$ and  $k\Leftarrow k+1$.
}

Record $K\Leftarrow k$ and $x^K:=x^k$.

\emph{(Phase II $-$ regularized Newton method)}

\While{the halting condition is not met}{
 Update regularized Newton direction ${d}^{k}$ by
\begin{eqnarray}\label{newton-direction}
    {d}^k=-( H^k+\beta\|g^{k}\| ^{\delta}I )^{-1}g^{k}.\end{eqnarray}
Find  the smallest nonnegative integer $j_k$  such that
\begin{eqnarray}\label{armijo-nm}f({x}^k+\alpha_2^{j_k}{d}^k)\leq{}f({x}^k)+
\mu_{2} \alpha_2^{j_k}\langle{}g^k,{d}^k\rangle.\end{eqnarray}
Set $\tau_k=\alpha_2^{j_k}, x^{k+1}= {x}^{k}+\tau_k{d}^{k}$ and $k\Leftarrow k+1$
}
 \caption{Gradient regularized Newton method (GRNM) \label{algorithm 0}}
\end{algorithm}
The framework of Algorithm \ref{algorithm 0} consists of two phases. The first phase exploits the gradient descent and aims to generate a good starting point  (i.e., $x^K$) for the second phase  from any arbitrary starting point $x^0$. {Phase I is terminated at point $x^K$ satisfying
\vspace{-1mm}
\begin{equation}
\label{halting-first}
\|g^K\|<{\epsilon}_G\qquad\text{and}\qquad k=K>K_0.
\end{equation}}
 The second phase employs the regularized Newton method for $k\geq K+1$ and could use the  halting condition similar to the first phase, namely,\vspace{-1mm}
\bea{c}\label{halting-cod}
\|g^k\|<{\epsilon}_N,
\eda
 where $\epsilon_N \in (0,\epsilon_G)$, because we will show that $\lim_{k\rightarrow\infty}\|g^k\|=0$ in Lemma \ref{sandle-point}.

\subsection{Convergence analysis}
Before the convergence analysis for  Algorithm \ref{algorithm 0}, we emphasize that if at a certain iteration $k$ point $x^k$ satisfies $\|g^k\| =0$, then we stop the algorithm immediately since $x^k$ is already a stationary point of problem \eqref{LSE-PQMR}. This indicates Algorithm \ref{algorithm 0} enables to terminate within finite steps, which is a very nice property. Therefore, hereafter, we always assume that  $\|g^k\|\neq0$ for any finite $k$ if no additional explanations are given.
\subsubsection{Convergence analysis for Phase I}
For  phase I, we show that $\{\tau_k:k\in[K]\}$ is bounded away from zero and  sequence $\{f(x^k):k\in[K]\}$ is non-increasing. Before that, for given point  $x^0\in\mathbb R^p$,  we define a function and two bounds
\bea{c}\label{def-phi-G}
~~~~~~\phi(x):=  \|x\| +\|\nabla f(x)\| ,~~~ \Phi(x^0):= {\sup}_{x\in  L(f(x^0))}  \phi(x),~~~
G(t) :=   {\sup}_{\|z\| \leq  t }  \|\nabla^2 f(z)\|,
\eda
where $L(f(x^0)$ is the level set given by \eqref{level set}.

\begin{lemma}\label{bound-B-G} If Assumption \ref{sample-RIP-general} holds, then  it has $G(\phi(x))\leq G(\Phi(x^0))<+\infty$ for any bounded $x^0\in\mathbb R^p$ and any $x\in L(f(x^0))$.
\end{lemma}
\begin{proof} The level set $L(f(x^0))$ is compact by Lemma \ref{obj-coer1} owing to the boundedness of $x^0$.  This and the continuity of $\phi(x)$ results in the boundedness of $\Phi(x^0)$, which by Lemma \ref{fgh-bound} shows  $G(\Phi(x^0))<+\infty$. Clearly, $G(\phi(x))\leq G(\Phi(x^0))$ due to $\phi(x)\leq \Phi(x^0)$ for any  $x\in L(f(x^0))$.
\end{proof}
\begin{lemma}\label{armi-exist-tau} Suppose Assumption \ref{sample-RIP-general} holds. Then there exists
\bea{c}
j_k \leq \max\left\lbrace 0, -\left\lfloor\log_{\alpha_1}\left( \frac{G(\Phi(x^0))}{2(1-\mu_{1} )}\right)\right\rfloor\right\rbrace  =:J_1<+\infty
\nnb\eda
to ensure  (\ref{armijo-sg}) for every $k\in[K]$, where $\lfloor a\rfloor$ represents the largest integer no greater than $a$.  
%
\end{lemma}
\begin{proof}For every $k\in[K]$, we take
\begin{eqnarray}
j_k= \left\{
             \begin{array}{lll}
              0, & \mbox{if}~~\nu_k:=\frac{G(\phi(x^k))}{2(1-\mu_{1} )} \leq 1,\\[1.25ex]
             -\left\lfloor\log_{\alpha_1}\left(\nu_k\right)\right\rfloor, & \mbox{otherwise}.
             \end{array}
        \right.
        \nnb
\end{eqnarray}
We prove the conclusion by induction.
For $k=0$,  as  $x^0\in L(f(x^0))$, Lemma \ref{bound-B-G} implies $G(\phi(x^0))\leq G(\Phi(x^0))<+\infty$.  If $ \nu_0  \leq 1$, then $j_0=0$ and $\tau_0=\alpha_1^{j_0}=1$. This results in
\begin{eqnarray}\label{lktauk-0}\tau_0G(\phi(x^0))  \leq2(1-\mu_{1} ).\end{eqnarray}
If $ \nu_0  > 1$, then taking $\tau_0=\alpha_1^{j_0}$ derives
\bea{c}
\tau_0=\alpha_1^{j_0}=\alpha_1^{-\left\lfloor\log_{\alpha_1}\left(\nu_0\right)\right\rfloor}
\leq\alpha_1^{- \log_{\alpha_1}\left(\nu_0\right) }= \frac{1}{\nu_0}= \frac{2(1-\mu_{1} )}{G(\phi(x^0))},
\nnb\eda
which also shows (\ref{lktauk-0}). Since $G(\phi(x^0))<+\infty$, it follows $\tau_0>0$. Now consider a point $x_t^0:= x^0 - t \tau_0  g^0$ for some $t\in(0,1)$. It can be shown that
\bea{c}
\|x_t^0\|\leq \|x^0\|+ t\tau_0 \|g^0 \|\leq \|x^0\|+ \|g^0 \|=\phi(x^0)
\nnb\eda
due to $\tau_0 \in(0,1]$. This indicates that
\bea{c}\label{xi-0-G}
\|\nabla^2f(x_t^0)\| \leq   \sup_{\|x\| \leq \phi(x^0)} \|\nabla^2 f(x)\|  = G(\phi(x^0)) \leq G(\Phi(x^0))<+\infty.
 \eda
Using this, $f$ being twice continuously differentiable and Taylor's expansion yields that
\bea{llll}\label{arm-half}
f(x^{0}-\tau_0 g^0 )
&=& f(x^{0})-\tau_0\langle g^0,g^0\rangle+({\tau_0^{2}}/{2})\langle g^0,\nabla^2f(\xi^0)g^0\rangle \\
&\leq&f(x^{0})- \tau_0\|g^0\|^2+({\tau_0^{2}}/{2})\|\nabla^2f(x_t^0)\|\|g^0\|^2 \\
&\leq&f(x^{0})- \tau_0\|g^0\|^2+({\tau_0^{2}}/{2})G(\phi(x^0)) \|g^0\|^2 &(\text{by \eqref{xi-0-G}})\\
&\leq&f(x^{0})- \tau_0\|g^0\|^2+ \tau_0(1-\mu_{1} )\|g^0\|^2 &(\text{by \eqref{lktauk-0}})\\
&=&f(x^{0})-\mu_{1}  \tau_0\|g^0\|^2,
\eda
displaying (\ref{armijo-sg}) for $k=0$. The above relation  also results in $f(x^{1}) \leq f(x^{0})$, thereby delivering $x^{1}\in L(f(x^0))$ and $G(\phi(x^1))\leq G(\Phi(x^0))<+\infty$ from Lemma \ref{bound-B-G}.

Now for $k=1$, same reasoning to check \eqref{lktauk-0} enables to render
\bea{llll}\label{lktauk-1}
\tau_1 G(\phi(x^1))  \leq2(1-\mu_{1} )\qquad\text{and}\qquad \tau_1>0.
\eda
Consider point $x_t^1:= x^1 - t \tau_1 g^1 $ for some $t\in(0,1)$. It follows that   \bea{llll}\|x_t^1\|\leq \|x^1\|+t\tau_1 \| g^1\|\leq \|x^1\|+ \| g^1\|=\phi(x^1)\nnb\eda
due to $\tau_1\in(0,1]$ and thus
\bea{llll}\label{xi-1-G}
\|\nabla^2f(x_t^1)\| \leq   \sup_{\|x\| \leq \phi(x^1)} \|\nabla^2 f(x)\|  = G(\phi(x^1)) \leq G(\Phi(x^0))<+\infty.
\eda
 Then combining \eqref{lktauk-1}, \eqref{xi-1-G} and the same reasoning to prove \eqref{arm-half} allows us to conclude  (\ref{armijo-sg}) for $k=1$.  Repeating the above step can show (\ref{armijo-sg}) for any $k\geq2.$

We can conclude from the above proof that $G(\phi(x^k)) \leq G(\Phi(x^0))<+\infty$ for any integer $k\geq0$. Hence,   $\nu_k\leq \overline \nu :=\frac{G(\Phi(x^0))}{2(1-\mu_{1} )} $, which brings out $-\left\lfloor\log_{\alpha_1}\left(\nu_k\right)\right\rfloor   \leq -\left\lfloor\log_{\alpha_1}\left(\overline \nu\right)\right\rfloor.$
Again, by $G(\Phi(x^0))<+\infty$ we have $ -\left\lfloor\log_{\alpha_1}\left(\overline \nu\right)\right\rfloor<+\infty$, leading to
\bea{l}
j_k \leq\max\left\{0,  -\left\lfloor\log_{\alpha_1}\left(\overline \nu\right)\right\rfloor\right\}=J_1<+\infty
,\nnb\eda
finishing the whole proof. 
\end{proof}
The above lemma implies that the step sizes in \eqref{armijo-sg} can be well defined. That is,  there is always a finite $j_k$ such that $\alpha_1^{j_k}>0$. This result
enables us to conclude that the first phase can stop within finitely many steps, as shown by the following lemma.
\begin{lemma}\label{first-alg-finite} If Assumption \ref{sample-RIP-general} holds, then the following statements are true.
\begin{itemize}[leftmargin=20pt]
\item[i)] Sequence $\{f(x^k):k\in[K]\}$ is non-increasing and $\{x^k:k\in[K]\}\subseteq L(f(x^0))$ is bounded.

\item[ii)] Halting condition (\ref{halting-first}) can be met with a finite $K$, namely, $K<+\infty$.
\end{itemize}
\end{lemma}
\begin{proof}
By (\ref{armijo-sg}), we have
\bea{c}
\label{non-increasing}
\alpha_1^{J_1}\|g^k \|^2 \leq  \mu_{1} \tau_k\|g^k \|^2\leq{}f(x^k)-f(x^{k+1}),
\eda
which together with $f(\cdot)\geq0$ yields that
\bea{c}
\sum_{k\geq 0} \mu_{1} \alpha_1^{J_1}\|g^k\|^2\leq\sum_{k\geq 0} (f(x^k)-f(x^{k+1}))=f(x^0)-f(x^{\infty})\leq{}f(x^0).
\nnb\eda
Then there is a nonnegative and finite integer $K$ such that $\|g^{K} \| <{\epsilon_G}.$  The sequence $\{f(x^k):k\in[K]\}$ is non-increasing due to \eqref{non-increasing} and thus $f(x^k)\leq f(x^0)$. This suffices to $\{x^k:k\in[K]\} \subseteq L(f(x^0))$, which by the boundedness of $L(f(x^0))$ yields the boundedness of $\{x^k:k\in[K]\}$.
\end{proof}

\subsubsection{Convergence analysis for Phase II}
Again,  similar to \eqref{def-phi-G}, for two given parameters $\beta>0$ and $\delta\in(0,1)$, we define
\bea{l}\label{def-psi-G}
\psi(x) :=   \|x\| + \beta^{-1}\|\nabla f(x)\|^{1-\delta} ,\qquad \Psi(x^0) := {\sup}_{x\in  L(f(x^0))}  \psi(x).
\eda
Same reasoning to prove Lemma \ref{bound-B-G} allows us to conclude that if $x^0$ is bounded, then
\bea{c}\label{def-psi-G-bound}
G(\psi(x))\leq G(\Psi(x^0))<+\infty, \qquad \forall~ x\in L(f(x^0)).
\eda

\begin{lemma}\label{armi-exist-n}
If Assumption \ref{sample-RIP-general} holds, then the following statements are true.
\begin{itemize}[leftmargin=20pt]
 \item[i)]  Sequence  $\{f( x^{k}):k\geq K\}$ is strictly decreasing, namely, $f( x^{k}) > f( x^{k+1})$ for any $x^k\neq x^{k+1}$.
 \item[ii)] Sequence $\{ {x}^{k}:k\geq K\}$ is bounded, namely, $\{ {x}^{k}:k\geq K\} \subseteq L(f(x^0)).$
 \item[iii)] If $\|g^{k}\| \geq \epsilon_N$ for a given $\epsilon_N>0$, then there exists
 \bea{l}\label{armijo-nm-eps}
 j_k \leq J_2 := \max\left\{0,-\left\lfloor \log_{\alpha_2}\left(  \frac{G(\Psi(x^0))}{2(1-\mu_{1} ) \beta \epsilon_N^{\delta}} \right) \right \rfloor\right\}  <+\infty,
\eda
to ensure (\ref{armijo-nm}) and hence,
\bea{c}\label{armijo-nm-eps-decrease}
f({x}^{k+1})\leq{}f({x}^k)+
\mu_{2} \alpha_2^{J_2}\langle g^k,{d}^k\rangle.\eda
 \end{itemize}
\end{lemma}
\begin{proof} i) It follows from \eqref{newton-direction} and $H^k\succeq0$  that
\bea{lll} \label{grad-d-0}
\langle g^{k},{d}^k\rangle =  -\langle (H^k+ \beta\|g^{k}\| ^{\delta}I) d^k, d ^k\rangle
 \leq    -\beta\|g^{k}\| ^{\delta} \|d^k\|^2
 < 0.
\eda
Therefore, we always have $j_k<+\infty$ such that (\ref{armijo-nm}), which derives $f( x^{k}) > f( x^{k+1})$.

ii) The assertion follows
$f( x^{k}) \leq f( x^{K})\leq f( x^0)$ for any $ k\geq K.$ immediately.

iii) For every $k\geq K$ such that $\|g^{k}\| \geq \epsilon_N$, we denote
\begin{eqnarray*}
             \begin{array}{lll}
\omega_k:=\frac{G(\psi({x}^k))}{2(1-\mu_{2} )\beta\|g^{k}\| ^{\delta}}\qquad {\rm and} \qquad\overline\omega:=\frac{G(\Psi(x^0))}{2(1-\mu_{1} ) \beta \epsilon_N^{\delta}}
             \end{array}
 \end{eqnarray*}
and take
\begin{eqnarray}\label{def-jk}
j_k=\left\{
             \begin{array}{lll}
              0, & \mbox{if}~\omega_k  \leq 1,\\[1.0ex]
             - \lfloor \log_{\alpha_2}(\omega_k)\rfloor, & \mbox{if}~\omega_k  > 1.
             \end{array}
        \right.
 \end{eqnarray}
 By \eqref{def-psi-G-bound} and ii), we obtain $G(\psi(x^k))\leq G(\Psi(x^0))<+\infty$, which by $\|g^{k}\| \geq \epsilon_N$ delivers $\omega_k   \leq \overline\omega  <+\infty.$
If $\omega_k\leq 1$, then $j_k=0$ and $\tau_k=\alpha_2^{j_k}=1$, thereby leading to
\begin{eqnarray}\label{lktauk}\tau_kG(\psi({x}^k))\leq{}2(1-\mu_{2} )\beta\|g^{k}\| ^{\delta}:=c_k.\end{eqnarray}
 If $\omega_k>1$, then from \eqref{def-jk},
\bea{ll}
\tau_k=\alpha_2^{j_k}=\alpha_2^{- \lfloor \log_{\alpha_2}(\omega_k)\rfloor}
\leq\alpha_2^{-  \log_{\alpha_2}(\omega_k) }= \frac{1}{\omega_k}=\frac{ c_k}{G(\psi({x}^k))},
\nnb\eda
which also leads to (\ref{lktauk}). Moreover, $\tau_k=\alpha_2^{j_k}>0$ due to $\omega_k<+\infty$.

 By   $H^k\succeq0$, \eqref{newton-direction} enables us to derive that
\bea{lll}
\|g^{k}\| =   \|(H^k+ \beta\|g^{k}\| ^{\delta}I){d}^k\|
  \geq   \beta\|g^{k}\| ^{\delta}\|d^k\|.
 \nnb\eda
The above relationship suffices to
\bea{c}\label{d-gradient}
\|d^k\|\leq \beta^{-1} \|g^{k}\| ^{1- \delta},
\eda
 thereby giving rise to
\begin{eqnarray}
\begin{array}{lllll}
\|x_t^k\| &= & \|{x}^k + t\tau_k {d}^k\| \\[1.2ex]
 & \leq& \|{x}^k\| +  \| {d}^k\|&(\text{by $t\in(0,1),0<\tau_k\leq \alpha_2<1$})\\[1.25ex]
 & \leq& \|{x}^k\| + \beta^{-1} \|g^{k}\| ^{1-\delta} &(\text{by \eqref{d-gradient}})\\[1.25ex]
 &=& \psi (x^k)  \leq  \Psi (x^0), &(\text{by \eqref{def-psi-G} and  $x^k \in L(f(x^0))$})
 \end{array} \nnb
\end{eqnarray}
where $x_t^k:={x}^k + t \tau_k  {d}^k $
 for some $t\in(0,1)$. This indicates that
 \bea{c}\label{xi-k-G}
\|\nabla^2f(x_t^k)\| \leq   \sup_{\|x\| \leq \psi(x^k)} \|\nabla^2 f(x)\|  = G(\psi(x^k)) \leq G(\Psi(x^0))<+\infty.
\eda
Direct calculation can yield the following relationships
\begin{eqnarray}\label{tau-grad-d}
\begin{array}{lllll}
\tau_k  \langle g^{k}, {d}^k \rangle
&=&
\mu_{2}\tau_k  \langle g^{k}, {d}^k \rangle+(1-\mu_{2})\tau_k  \langle g^{k}, {d}^k \rangle\\[1.25ex]
&\leq&\mu_{2} \tau_k  \langle g^{k}, {d}^k \rangle-(1-\mu_{2} )\tau_k  \beta\|g^{k}\| ^{\delta} \|d^k\|^2 ~~&(\text{by \eqref{grad-d-0}})\\[1.25ex]
&=&\mu_{2} \tau_k  \langle g^{k}, {d}^k \rangle
-  (c_k  { \tau_k }/{2})\| {d}^k \| ^2.&(\text{by \eqref{lktauk}})
 \end{array}
\end{eqnarray}
 Now we conclude from Taylor's expansion   that
\bea{llll}\label{arm-half-k}
f({x}^k +\tau_k  {d}^k )
&=&f({x}^k )+
\tau_k  \langle g^{k}, {d}^k \rangle+({ \tau_k ^2}/{2})\langle{}\nabla^2f(x_t^k ) {d}^k , {d}^k \rangle \\
&\leq&f({x}^k )+
\mu_{2}\tau_k \langle g^{k}, {d}^k \rangle -  (c_k  { \tau_k }/{2})\| {d}^k \| ^2 + ({ \tau_k ^{2}}/{2})\langle{}\nabla^2f(x_t^k ) {d}^k , {d}^k \rangle &(\text{by \eqref{tau-grad-d}})\\
&\leq&f({x}^k )+\mu_{2} \tau_k \langle g^{k}, {d}^k \rangle
- ( c_k  - \tau_k  G(\psi(x^k))   )({ \tau_k }/{2})\| {d}^k \| ^2 &(\text{by \eqref{xi-k-G})} \\
&\leq&f({x}^k )+\mu_{2} \tau_k \langle g^{k}, {d}^k \rangle. &(\text{by \eqref{lktauk}})
\eda
resulting in (\ref{armijo-nm}). Finally,
it follows from $\omega_k   \leq \overline\omega  <+\infty$ that $\left\lfloor\log_{\alpha_2}\left(\omega_k\right)\right\rfloor
 \geq \left\lfloor\log_{\alpha_2}\left(\overline\omega\right)\right\rfloor$,
because of this and \eqref{def-jk}, we obtain
\bea{c}
 j_k\leq  \max\left\{0, -\left\lfloor\log_{\alpha_2}\left(\overline\omega\right)\right\rfloor\right\}=J_2<+\infty,
 \nnb\eda
displaying  \eqref{armijo-nm-eps}. Inequality \eqref{armijo-nm-eps-decrease} follows $\tau_k\geq \alpha_2^{J_2}$ and $\langle g^{k},{d}^k\rangle  < 0$ immediately.
\end{proof}
\noindent Now we are ready to claim that the gradient eventually vanishes by the following lemma.

\begin{lemma}\label{sandle-point}If Assumption \ref{sample-RIP-general} holds, then  $\lim_{k\to\infty}\|g^{k}\| =\lim_{k\to\infty} \| {x}^{k+1}- {x}^k\| =0.$
\end{lemma}
\begin{proof} Suppose $\lim_{k\to\infty}\|g^{k} \| \neq 0$, namely,  $\limsup \|g^{k}\| >0$.
Let
\bea{c}
{\epsilon}:=\frac{1}{2}\limsup_{k\to\infty}\|g^{k}\|, \qquad\mathcal{I}_{\epsilon}(k):=\left\{\ell\in\{K,K+1,\ldots,k\}:  \|g^{\ell}\| \geq{\epsilon}\right\},
\nnb\eda
and $|\mathcal{I}_{\epsilon}(k)|$ denote the cardinality of set $\mathcal{I}_{\epsilon}(k)$.  Than we have
\begin{eqnarray} \label{fact-I-infty}\lim_{k\to\infty}|\mathcal{I}_{\epsilon}(k)|=\infty.\end{eqnarray}
due to  $\limsup \|g^{k}\| >0$.  It follows from Lemma \ref{armi-exist-n} 1) that, for any $\ell\geq K$, \begin{eqnarray} \label{fact-decrease}f(x^{\ell})\geq f(x^{\ell+1}).\end{eqnarray}
Now for any $\ell\in\mathcal{I}_{\epsilon}(k)$, it has $\|g^{\ell}\| \geq{\epsilon}$, which allows us to derive that
\bea{llll} \label{fact-decrease-d}
f({x}^{\ell+1})&\leq&f({x}^\ell)+
\mu_{2} \alpha_2^{J_2}\langle g^\ell,{d}^\ell\rangle
&(\text{by \eqref{armijo-nm-eps-decrease}})\\
&\leq&f({x}^\ell)
- \mu_{2} \alpha_2^{J_2}\beta\|g^{l} \| ^{\delta}\| {d}^\ell\| ^2\qquad&(\text{by \eqref{grad-d-0}})\\
&\leq&f({x}^\ell)
- \mu_{2} \alpha_2^{J_2} \beta^{-1}\|g^\ell\| ^{2-\delta}&(\text{by \eqref{d-gradient}})\\
&\leq&f({x}^\ell)
- \mu_{2} \alpha_2^{J_2} \beta^{-1} \epsilon^{2-\delta}.
\eda
Using these facts enables to generate the following chain of inequalities
\bea{llll}
f( {{x}}^{K})-f( {{x}}^{k+1})&=&\sum_{\ell=K}^{k} (f( {{x}}^{\ell})-f( {{x}}^{\ell+1}) )\\
&\geq&\sum_{\ell\in\mathcal{I}_{\epsilon}(k)} (f( {{x}}^{\ell})-f( {{x}}^{\ell+1}) )\qquad~~(\text{by \eqref{fact-decrease}})\\
&\geq&\sum_{\ell\in\mathcal{I}_{\epsilon}(k)}\mu_{2} \alpha_2^{J_2} \beta^{-1} \epsilon^{2-\delta}  \qquad\qquad(\text{by \eqref{fact-decrease-d}}) \\
&=&|\mathcal{I}_{\epsilon}(k)|\mu_{2} \alpha_2^{J_2} \beta^{-1} \epsilon^{2-\delta}.
\nnb\eda
We then conclude from \eqref{fact-I-infty} and the above inequality that
\bea{c}\lim_{k\to\infty}f( {x}^{k+1})=-\infty.\nnb\eda
This contradicts $f(x)\geq0$, and hence $\lim_{k\to\infty}\|g^{k} \| = 0$. Again by \eqref{d-gradient},  it follows \bea{c} \| {x}^{k+1}- {x}^k\|
=\tau_k\| {d}^k\| \leq \alpha_2 \beta^{-1}\|g^{k} \| ^{1-\delta} \rightarrow 0,
\nnb\eda
proving the whole lemma.
\end{proof}

\subsubsection{Convergence properties for Algorithm \ref{algorithm 0}}
In this subsection, we provide the main convergence results, including the sequence convergence and the convergence rate. 
{ To proceed with that, we need the following lemma from \cite{Absil}.
\begin{lemma}[Theorem 3.2, in \cite{Absil}]\label{who-con-lemma}
Let \( h : \mathbb{R}^p \to \mathbb{R} \) be an analytic function and  sequence  \(\{\tilde{x}_{k}\}\) satisfy the following strong descent condition
\bea{lll}
f(x^k)-f(x^{k+1})\geq\mu\|g^k \|\|x^{k+1}-x^{k}\|
\nnb\eda
for all $k$ and for some $\mu>0$.
Then, either $\lim_{k\to\infty}\|\tilde{x}_k\|=+\infty$ or there exists a single point $\tilde{x}$ such that $\lim_{k\to\infty}\tilde{x}_k=\tilde{x}.$
\end{lemma}
The above lemma enables us to derive the whole sequence convergence which does not rely on the initial points, as stated below.
\begin{theorem}[Sequence convergence]\label{convergence}Suppose that Assumption \ref{sample-RIP-general} holds.
Let $\{ {x}^{k}:k\geq0\}$ be the sequence generated by Algorithm \ref{algorithm 0}. Then the whole sequence converges to a stationary point.
\end{theorem}
\begin{proof}Lemmas \ref{first-alg-finite} and \ref{armi-exist-n}  showed that sequence $\{ {x}^{k}: k\geq0\}  \subseteq L(f(x^0))$ and Lemma \ref{obj-coer1} stated $L(f(x^0))$ is compact. Therefore, there exists a  subsequence of $\{ {x}^{k}\}$ that converges to $\bar{x}$.  By Lemma \ref{sandle-point}, we have $\nabla{}f( \bar{x})=0$. It is clear that  function $f$ is real polynomial, hence real analytic. To get the whole sequence convergence, based  on Lemma \ref{who-con-lemma} it suffices to prove the so-called strong descent condition. Notice that
 inequality (\ref{armijo-sg}) implies that
\bea{lll}
f(x^k)-f(x^{k+1})\geq\mu_{1} \alpha_1^{j_k}\|g^k \| ^2=\mu_{1}\|g^k \| \|x^{k+1}-x^{k}\|,~~ k=1,2,\ldots,K.
\nnb\eda
Now we conclude that
\bea{llll}
\tau_k\langle -g^k,{d}^k\rangle&=&\tau_k\langle g^k,(H^k+\beta\|g^{k}\| ^{\delta}I )^{-1}g^{k}\rangle\\
&\geq&\tau_k\beta^{-1}\|g^{k}\|^{2-\delta} \qquad\qquad\qquad\qquad~&(\text{by $\lambda_{\mathrm{min}}(H^k+\beta\|g^{k}\|^{\delta}I)^{-1}\geq\beta^{-1}\|g^{k}\|^{-\delta}$})\\
&\geq&\tau_k\|g^{k}\| \|d^{k}\|=\|g^k \| \|x^{k+1}-x^{k}\|~&(\text{by \ref{d-gradient}})
\nnb\eda
which together with \eqref{armijo-nm-eps-decrease} yields that
\bea{lll}
f(x^k)-f(x^{k+1})\geq\mu_2\|g^k \| \|x^{k+1}-x^{k}\|,~~k=K+1,K+2,\ldots.
\nnb\eda
Therefore, the strong descent condition holds with $\mu=\min\{\mu_1,\mu_2\}$, completing the proof.
\end{proof}}

The following result shows that after Algorithm \ref{algorithm 0} generates a solution  $\bar{x}$, which  necessarily is a stationary point, we can check a simple condition to see whether it satisfies the positive definiteness of $\nabla^2f( \bar{x})$. To proceed with that, we observe from Assumption  \ref{sample-RIP-general} that, for any $x$ and $v$,
\bea{llll} \label{min-lambda-H}
 \left\langle  v, \frac{2}{n}{\bf A}(x)^T{\bf A}(x) v \right\rangle  &=&\frac{2}{n}\sum_{i=1}^n \left\langle  v,  A_ix  \right\rangle^2  \geq  2{\lambda}\|v\|^2\|x\|^2.
\eda
%

\begin{lemma}\label{stepsize-equ1}Suppose that Assumption  \ref{sample-RIP-general} holds.
Let $\{ {x}^{k}:k\geq0\}$ be the sequence generated by Algorithm \ref{algorithm 0} with $\mu_2\in(0,1/2)$. If  {  limit point} $ \bar{x}$ satisfies
\bea{lll}\label{hessian-min} \|\frac{1}{n}\sum_{i=1}^n \varphi_i(\bar{x})A_i\| \leq \left(2-\frac{8}{5-4\mu_2^2} \right){\lambda} \|\bar x\|^2,\eda
   then for sufficiently large $k$,  $\tau_k=1$ in Phase II of  Algorithm \ref{algorithm 0}.
\end{lemma}
\begin{proof}
For notational simplicity, denote \bea{llll}\label{step-size-1-fact-00}
a_1&:=& 2-\frac{8}{5-4\mu_2^2}\in(0,1)&(\text{by $\mu_2\in(0,1/2)$})\\
a_2&:=& \left(\frac{1}{2}-\mu_2\right)(1+\mu_2)(2-a_1)-a_1\\
a_3&:=& 1-\mu_2^2\in(0,1).&(\text{by $\mu_2\in(0,1/2)$})
\eda
Direct verification enables to show that
\begin{equation}\label{step-size-1-fact-01}
\begin{array}{l}
a_2= \left( \frac{1}{2}-\mu_2\right) (1+\mu_2)(2-a_1)-a_1 
> \left( \frac{1}{2}-\mu_2\right) \left( \frac{1}{2}+\mu_2\right) (2-a_1)-a_1 =0.
\end{array}\end{equation} 
{Notice that  $ \bar{x}$ is the limit point of $\{x^k\}$},  which together with $a_2>0$ delivers that
\bea{llll}\label{step-size-1-fact-0}
{\eta}\|x^k+\bar{x}\|\|x^k-\bar{x}\|\leq \frac{a_2{\lambda} \|\bar x\|^2}{n}
\eda
for sufficiently large $k$. Moreover, it follows from $\varphi_i({x}^k)=\varphi_i(\bar{x})+ \langle x^k-\bar{x},A_i(x^k+\bar{x})\rangle$ that
\bea{llll}\label{step-size-1-fact-4}
&&|\langle d^k,\frac{1}{n}\sum_{i=1}^n \varphi_i({x}^k)A_id^k\rangle|\\
&\leq&|\langle d^k,\frac{1}{n}\sum_{i=1}^n \varphi_i(\bar{x})A_id^k\rangle|+|\frac{1}{n}\sum_{i=1}^n\langle x^k-\bar{x},A_i(x^k+\bar{x})\rangle \langle d^k,A_id^k\rangle| \\
&\leq&|\langle d^k,\frac{1}{n}\sum_{i=1}^n \varphi_i(\bar{x})A_id^k\rangle|+\sqrt{\frac{1}{n}\sum_{i=1}^n\langle x^k-\bar{x},A_i(x^k+\bar{x})\rangle^2} \sqrt{\frac{1}{n}\sum_{i=1}^n\langle d^k,A_id^k\rangle^2} \\
&\leq&|\langle d^k,\frac{1}{n}\sum_{i=1}^n \varphi_i(\bar{x})A_id^k\rangle|+ {\eta}\|x^k+\bar{x}\|\|x^k-\bar{x}\|\|d^k\|^2
\qquad (\text{by Assumption \ref{sample-RIP-general}})\\
&\leq& \frac{a_1{\lambda} }{n} \|\bar x\|^2\|d^k\|^2+\frac{a_2{\lambda} }{n}  \|\bar x\|^2\|d^k\|^2\qquad\qquad (\text{by \eqref{hessian-min}  and \eqref{step-size-1-fact-0}})\\
&=&  \frac{(a_1+a_2){\lambda} }{n} \|\bar x\|^2\|d^k\|^2.
\eda
Therefore, by $\nabla^2f(x^k)=H^k+\frac{1}{n}\sum_{i=1}^n \varphi_i({x}^k)A_i$ and the above condition, we get
\bea{lllll}\label{step-size-1-fact-2}
\langle g^{k},d^k\rangle&=&\mu_2\langle g^{k},d^k\rangle-(1-\mu_2)\langle -g^{k},d^k\rangle\\
&=&\mu_2\langle g^{k},d^k\rangle-(1-\mu_2)\langle  (H^k+\beta\|g^{k}\|^\delta I)d^k,d^k\rangle\qquad (\text{by \eqref{newton-direction}})\\
&=&\mu_2\langle g^{k},d^k\rangle-(1-\mu_2)\langle (\nabla^2f(x^k)- \frac{1}{n}\sum_{i=1}^n \varphi_i({x}^k)A_i+\beta\|g^{k}\|^\delta I)d^k,d^k\rangle\\
&\leq&\mu_2\langle g^{k},d^k\rangle-(1-\mu_2)\langle (\nabla^2f(x^k)+\beta\|g^{k}\|^\delta I)d^k,d^k\rangle+ \frac{(1-\mu_2)(a_1+a_2){\lambda} }{n}\|\bar x\|^2\|d^k\|^2.
\eda
Moreover, by the definition of $H^k$ and \eqref{min-lambda-H}, we have
\bea{llll} \label{min-lambda-Hk}
\lambda_{\min}(H^k) = \inf_{\|v\|=1}\left\langle  v, \frac{2}{n}{\bf A}(x^k)^T{\bf A}(x^k) v \right\rangle  \geq  2{\lambda} \|x^k\|^2,
\eda
which together with \eqref{newton-direction} enables us to derive that
\bea{lll}\label{gk-lower-bound}
\|g^{k}\| &=&   \|(H^k+ \beta\|g^{k}\| ^{\delta}I){d}^k\|\\
  &\geq & \lambda_{\min}(H^k+ \beta\|g^{k}\| ^{\delta}I) \|{d}^k\| \\
   &\geq&  \max\{\lambda_{\min}(H^k), \beta\|g^{k}\| ^{\delta}\} \|{d}^k\| \\
    &\geq& \max\{ 2{\lambda}\|x^k\|^2, \beta\|g^{k}\| ^{\delta}\} \|{d}^k\|.
\eda
This results in  $\lim_{k\to \infty} \|{d}^k\| \leq \lim_{k\to \infty}\beta^{-1}\|g^{k}\|^{1-\delta}  =0$,
and thus for sufficiently large $k$,
 \bea{llll}\label{step-size-1-fact-50}
&&2(1-\mu_2)\beta\|g^{k}\|^\delta - 3{\eta}(2\|x^k\|+\|d^k\|)\|d^k\|\\
&\geq&2(1-\mu_2)\beta (2{\lambda} \|x^k\|^2\|{d}^k\|)^\delta- 3{\eta}(2\|x^k\|+\|d^k\|) \|d^k\|~~(\text{by \eqref{gk-lower-bound}})\\
&=&\left[ 2(1-\mu_2)\beta (2{\lambda} \|x^k\|^2)^\delta  - 3{\eta}(2\|x^k\|+\|d^k\|) \|d^k\|^{1-\delta} \right]  \|{d}^k\|^\delta\\
&\geq&0.~~(\text{by $x^k \to \bar x \neq 0$ and $d^k \to 0$})
\eda
 Now for any given $t\in(0,1)$,  using the above condition  yields
\bea{llll}\label{step-size-1-fact-5}
 &&\frac{1}{2} \langle d^k,  \big(\nabla^2f(x^k+t d^k)-\nabla^2f(x^k)\big)d^k\rangle\\
&\leq& \frac{1}{2} \|\nabla^2f(x^k+t d^k)-\nabla^2f(x^k)\|\|d^k\|^2\\
&\leq& \frac{3}{2} {\eta}(\|x^k+t d^k\|+\|x^k\|)\|x^k+td^k-x^k\|\|d^k\|^2&(\text{by Lemma \ref{hessian-lip}})\\
&\leq& \frac{3}{2} {\eta}(2\|x^k\|+\|d^k\|) \|d^k\|^3 \\
&\leq& (1-\mu_2)\beta\|g^{k}\|^\delta\|d^k\|^2. &(\text{by \eqref{step-size-1-fact-50}})\eda
{Since $0<a_1<1$, we conclude from condition \eqref{hessian-min} that for any $ v(\neq0)\in\mathbb{R}^p$,
\bea{llll}\label{cor-positive}
 \langle v, \nabla^2f( \bar{x}) v \rangle
&=&\frac{2}{n}\langle v,  {\bf A}(\bar{x})^T{\bf A}(\bar{x}) v\rangle
+ \frac{1}{n} \langle v,  \sum_{i=1}^n \varphi_i(\bar{x})A_iv\rangle&~~(\text{by \eqref{grad-hess}})\\
&\geq& {2{\lambda}}  \|v\|^2\|\bar x\|^2
- \frac{1}{n} \|v\|^2 \|\sum_{i=1}^n \varphi_i(\bar{x})A_i\|&~~(\text{by \eqref{min-lambda-H}})\\
&=&\|v\|^2 \left(2{\lambda} \|\bar x\|^2 - \|\frac{1}{n}\sum_{i=1}^n \varphi_i(\bar{x})A_i\|\right)
>0.&~~(\text{by \eqref{hessian-min}})\\
\eda}
To conclude the conclusion, we need the last fact:
 \bea{llll}\label{step-size-1-fact-6}
 \lambda_{\mathrm{min}}(\nabla^2f(x^k))&\geq& a_3\lambda_{\mathrm{min}}(\nabla^2f(\bar{x}))& (\text{by $\lim_{k\rightarrow \infty}x^k=\bar{x}$  and $ a_3\in(0,1)$})\\
 &\geq&\frac{ a_3}{n} (2{\lambda}  \|\bar x\|^2
-  \|\sum_{i=1}^n \varphi_i(\bar{x})A_i\|)&(\text{by \eqref{cor-positive}})\\
& >& \frac{a_3}{n}(2 {\lambda}  \|\bar x\|^2
- a_1{\lambda}  \|\bar x\|^2)&(\text{by \eqref{hessian-min}})\\
& = & \frac{  a_3(2-a_1){\lambda}  }{n} \|\bar x\|^2>0. & (\text{by $a_1\in(0,1)$})
\eda
 Now, we have the following chain of inequalities
\bea{llll}
&& f( x^k+d^k)-f(x^k)\\
&=&\langle g^{k},d^k\rangle
+ \frac{1}{2} \langle d^k,  \nabla^2f(x^k+t d^k)d^k\rangle\\
&=& \langle g^{k},d^k\rangle+\frac{1}{2} \langle d^k,  \nabla^2f(x^k)d^k\rangle
+ \frac{1}{2} \langle d^k,  \big(\nabla^2f(x^k+t d^k)-\nabla^2f(x^k)\big)d^k\rangle\\
&\leq& \langle g^{k},d^k\rangle+\frac{1}{2} \langle d^k,  \nabla^2f(x^k)d^k\rangle + (1-\mu_2)\beta\|g^{k}\|^\delta\|d^k\|^2
 ~&(\text{by \eqref{step-size-1-fact-5}}) \\
&\leq& \mu_2\langle g^{k},d^k\rangle-(\frac{1}{2} -\mu_2)\langle \nabla^2f(x^k)d^k,d^k\rangle+\frac{(1-\mu_2)(a_1+a_2){\lambda} }{n} \|\bar x\|^2 \|d^k\|^2 ~&(\text{by \eqref{step-size-1-fact-2}})\\
&\leq& \mu_2\langle g^{k},d^k\rangle-\left[ (\frac{1}{2} -\mu_2) a_3(2-a_1) -(1-\mu_2)(a_1+a_2)\right] \frac{ {\lambda}}{n}   \|\bar x\|^2\|d^k\|^2 ~&(\text{by \eqref{step-size-1-fact-6}})\\
&=& \mu_2\langle g^{k},d^k\rangle,
\nnb\eda
where the last equality used the fact that
\bea{llll}
&& \left( \frac{1}{2} -\mu_2\right)  a_3(2-a_1)-   (a_1+a_2)(1-\mu_2)\\
&=&(1-\mu_2)\left[ \left( \frac{1}{2} -\mu_2\right) (1+\mu_2)(2-a_1)-   (a_1+a_2)\right]\\
&=&0.\nnb
 \eda
 The above condition indicates that $\tau_k=1$ for sufficiently large $k$.
\end{proof}
The above lemma means that full Newton step size is eventually taken, namely, $x^{k+1}=x^k+d^k$ for sufficiently large $k$. This property enables us to derive the following convergence with rate $1+\delta$, which is at least superlinear.
\begin{theorem}[Convergence rate]\label{convergence-rate}Suppose that Assumption  \ref{sample-RIP-general} holds.
If  any one of its nonzero accumulating points $ \bar{x}$ of  the sequence generated by Algorithm \ref{algorithm 0} with  $ \mu_2\in(0,1/2)$ satisfies (\ref{hessian-min}),  then there are constants $c_1>0, c_2>0$ and $c_3\in(0,1/2)$ such that for sufficiently large {$k>K$},
\bea{llll}\label{super-linear-bound}
 \|x^{k+1}-\bar{x}\|\leq
 c_1\|x^k-\bar{x}\|^2+  c_2\|x^k-\bar{x}\|^{1+\delta}   +c_3\|x^k-\bar{x}\|.
 \eda
If
$\varphi_i(\bar{x})=0$ for all $i\in[n]$, then the whole sequence converges to $\bar{x}$ with rate $1+\delta$.
\end{theorem}
\begin{proof}  {When $k>K$, the algorithm enters into phase II.} We note that
\bea{lllll}\label{def-theta-k}
\Theta^k:=H^k+\beta\|g^{k}\|^\delta I = \nabla^2f(x^k)-\frac{1}{n}\sum_{i=1}^n \varphi_i(x^k)A_i +\beta\|g^{k}\|^\delta I.
\eda
 { Condition \eqref{min-lambda-Hk}   results in}
\bea{llll}\lambda_{\min}(H^k)&\geq& 2{\lambda} \|x^k\|^2 \geq 2a_1 {\lambda} \|\bar x\|^2
\nnb\eda
due to $0<a_1<1$, where $a_1$ is defined as \eqref{step-size-1-fact-00}.  As a consequence, for any sufficiently large $k$,
\bea{llll}\label{linear-rate-1}
\|(\Theta^k)^{-1}\| \leq \frac{1}{\lambda_{\min}(H^k) + \beta\|g^{k}\|^\delta}   \leq \frac{1}{2a_1 {\lambda} \|\bar x\|^2} =:C.
\eda
Denote $\varpi^k:=x^k-\bar{x}$. It follows from the fact
\bea{llll}\varphi_i({x}^k)&=&\langle{x}^k, A_ix^k\rangle-\langle\bar{x}, A_i\bar{x}\rangle+\langle\bar{x}, A_i\bar{x}\rangle-b_i
=\langle \varpi^{k}, A_i(x^k+\bar{x})\rangle+\varphi_i(\bar{x}),
\nnb\eda
and $\nabla f(\bar{x})=0$ that
 \bea{llll}\label{int-bound}
H^k \varpi^{k}- g^k &=& H^k \varpi^{k}- g^k -\nabla f(\bar x) \\
&=&\frac{1}{n}\left( \sum_{i=1}^n 2A_ix^k(x^k)^TA_i\varpi^{k} -\varphi_i(x^k)A_ix^k-\varphi_i(\bar x)A_i\bar x\right) \\
&=&\frac{1}{n}\left( \sum_{i=1}^n 2\langle \varpi^k, A_ix^{k} \rangle A_ix^k -(\varphi_i(x^k)-\varphi_i(\bar x)) A_ix^k-\varphi_i(\bar x)A_i(x^k-\bar x)\right) \\
&=&\frac{1}{n}\left( \sum_{i=1}^n 2\langle \varpi^k, A_ix^{k} \rangle A_ix^k -\langle \varpi^{k}, A_i(x^k+\bar{x})\rangle A_ix^k-\varphi_i(\bar x)A_i\varpi^{k}\right) \\
 &=&\frac{1}{n}\left( \sum_{i=1}^n  \langle \varpi^k, A_i \varpi^k \rangle A_ix^k  -\varphi_i(\bar x)A_i\varpi^{k}\right),
 \eda
 Moreover, direct verification enables us to obtain
 \bea{llll}\label{skbound}
\|\sum_{i=1}^n  \langle \varpi^k, A_i \varpi^k \rangle A_i \| &=& \sup_{\|u\|=1} |\langle u,\frac{1}{n}\sum_{i=1}^n  \langle \varpi^k, A_i \varpi^k \rangle A_i u\rangle|\\
&=& \sup_{\|u\|=1} |\frac{1}{n}\sum_{i=1}^n\langle\varpi^{k}, A_i \varpi^{k} \rangle \langle u, A_iu\rangle |\\
&\leq&\sup_{\|u\|=1} \sqrt{\frac{1}{n}\sum_{i=1}^n\langle\varpi^{k}, A_i \varpi^{k} \rangle^2}\sqrt{ \frac{1}{n}\sum_{i=1}^n\langle u, A_iu\rangle^2}\\
&\leq&\sup_{\|u\|=1} {\eta} \|\varpi^{k}\|^2\|u\|^2 ~~~ (\text{by Assumption \ref{sample-RIP-general}})\\
&=&{\eta} \|\varpi^{k}\|^2.\nnb\eda
Using the above two facts and  $\|{x}^k\|\leq \|\bar{x}\|  +\|\varpi^k\| \leq 2\|\bar{x}\|$ by $\varpi^{k}\rightarrow 0$ enable us to derive that
  \bea{llll}\label{int-bound}
\|H^k \varpi^{k}- g^k\| &\leq& {\eta} \|\varpi^{k}\|^2\|x^k\|  +\|\frac{1}{n}\sum_{i=1}^n\varphi_i(\bar x)A_i\varpi^{k}\|\\
&\leq& 2{\eta}\|\bar x\| \|\varpi^{k}\|^2+\|\frac{1}{n}\sum_{i=1}^n\varphi_i(\bar x)A_i\|\|\varpi^{k}\|.
 \eda
Next, we estimate $\|g^k\|$. To proceed with that, one can show
\bea{llll}\label{gk-bound-1}
\|H^k \| &=& \sup_{\|u\|=1} |\langle u,\frac{2}{n}\sum_{i=1}^n    A_i x^k (x^k)^T A_i u\rangle|\\
&=& \sup_{\|u\|=1} |\frac{2}{n}\sum_{i=1}^n\langle u, A_i x^{k} \rangle^2 |\\
&\leq&\sup_{\|u\|=1} 2 {\eta} \|x^{k}\|^2\|u\|^2 ~~~ (\text{by Assumption \ref{sample-RIP-general}})\\
&\leq&3{\eta} \|\bar x\|^2.\qquad\qquad\qquad\quad~ (\text{by $x^{k} \to \bar x$})  \eda
Now, we derive $\|g^k\|$ from the above two facts that
  \bea{llll}\label{gk-bound}
\|g^k\| &\leq& \|H^k \varpi^{k}\|+ \|H^k \varpi^{k}- g^k\|  \\
&\leq& \|H^k\|\| \varpi^{k}\|+   2{\eta}\|\bar x\| \|\varpi^{k}\|^2+\|\frac{1}{n}\sum_{i=1}^n\varphi_i(\bar x)A_i\|\|\varpi^{k}\| ~ &(\text{by \eqref{skbound}})\\
&\leq& 3{\eta} \|\bar x\|^2 \| \varpi^{k}\|+  2{\eta}\|\bar x\| \|\varpi^{k}\|^2+ a_1{\lambda} \|\bar x\|^2\|\varpi^{k}\|~ &(\text{by \eqref{gk-bound-1}, \eqref{hessian-min}})\\
&\leq& 6{\eta} \|\bar x\|^2 \| \varpi^{k}\|.~ &(\text{by $\varpi^{k}\to 0$ and ${\lambda} \leq {\eta}$})
 \eda
By  Lemma \ref{stepsize-equ1} that $x^{k+1}= x^{k} +\tau_kd^k=x^{k} + d^k$ for sufficiently large $k$, we have 
\bea{llll}\label{def-c3}
 \|\varpi^{k+1}\|&=&\|d^k+\varpi^{k}\|=\|(\Theta^k)^{-1}(-g^k+\Theta^k \varpi^{k})\|\\
 &=&\left\|(\Theta^k)^{-1}\left( H^k \varpi^{k} - g^k+  \beta\|g^{k}\|^\delta   \varpi^{k}\right) \right\|~&(\text{by \eqref{def-theta-k}})\\
&\leq& C\left( \|H^k \varpi^{k} - g^k\|+  \beta\|g^{k}\|^\delta   \|\varpi^{k}\|\right) ~&(\text{by \eqref{linear-rate-1}}) \\
  &\leq& 2{\eta}C\|\bar{x}\|\|\varpi^{k}  \|^2+  C\beta \|g^{k}\|^\delta   \|\varpi^{k}\|+C\|\frac{1}{n}\sum_{i=1}^n \varphi_i(\bar{x})A_i\|\| \varpi^{k}\|~&(\text{by \eqref{int-bound}})\\
   &\leq& c_1 \|\varpi^{k}  \|^2+  c_2   \|\varpi^{k}\|^{1+\delta}+C\|\frac{1}{n}\sum_{i=1}^n \varphi_i(\bar{x})A_i\|\| \varpi^{k}\|.~&(\text{by \eqref{gk-bound}})
 \eda
where $c_1:=2{\eta}C\|\bar{x}\|$ and $ c_2:=C\beta (6{\eta} \|\bar x\|^2)^\delta$.  By \eqref{linear-rate-1} and \eqref{hessian-min}, one can see that
\bea{llll}
c_3:=C\|\frac{1}{n}\sum_{i=1}^n \varphi_i(\bar x)A_i\|\leq \frac{a_1 {\lambda} \|\bar x\|^2}{2a_1 {\lambda} \|\bar x\|^2}
=\frac{1}{2}.\nnb\eda
If $\varphi_i(\bar{x})=0$, then relation (\ref{hessian-min}) holds  and
\bea{llll}
\lim_{k\rightarrow\infty} \frac{\|\varpi^{k+1}\|}{\|\varpi^{k}\|^{1+\delta}}& \leq & 2{\eta}C\|\bar{x}\| + c_2,
\nnb\eda
 which shows that the convergence rate is $1+\delta$.
\end{proof}

\section{Oracle Convergence Analysis for GRNM}\label{sec:oracle-con}

As mentioned in Section \ref{sec:related works}, \cite{Thaker} stated that problem (\ref{LSE-PQMR}) has no spurious local minima with a high probability when $\{A_1,\ldots,A_n\}$ is a set of symmetric Gaussian random matrices.
An very interesting question that arises is whether the solver calculated by Algorithm \ref{algorithm 0} is fairly close to the true signal.
Inspired by  the work in \cite{ge17} and \cite{Thaker}, we study the relation between local minima and global minima when $\{A_1,\ldots,A_n\}$ satisfies Assumption \ref{sample-RIP-general}. Particularly, we will prove that the solver implemented by Algorithm \ref{algorithm 0} nears to true value $x^*$ under some mild conditions.

{
\begin{lemma}[Theorem 2.3, in \cite{tru22}]\label{gd-ami}
Let \( h : \mathbb{R}^p \to \mathbb{R} \) be a continuously differentiable function which satisfies the Kurdyka-Lojasiewicz inequality. For any \( w^0 \in \mathbb{R}^p \), construct sequence {$\{w^k:k\geq0\}$} by {\( w^{k+1} = w^k - \tau_k\nabla h(w^k) \)}, where step size $\tau_k$ is chosen by Armijo' rule. Let the hyperparameters in this gradient descent algorithm be randomly chosen.  Then, sequence {$\{w^k:k\geq0\}$} either diverges to infinity
or converges to a stationary point of $h$ which is not a strict saddle
point.
\end{lemma}
\begin{lemma}\label{convergence-local}Suppose that Assumption \ref{sample-RIP-general} holds, ${\eta}<3{\lambda}$, and $\kappa_*\leq {(3{\lambda}-{\eta})r_*^2}/{7}$.
Let $\{ {x}^{k}\}$ be the sequence generated by Algorithm \ref{algorithm 0} with setting a sufficiently large (but finite) $K_0$. Then the whole sequence converges to a unique local minimizer in $\mathcal{N}_*$.
\end{lemma}
\begin{proof} If we set $K_0=\infty$, then we get an infinite sequence {$\{x^k:k\geq0\}$}. To differ from the sequence generated by GRNM, we denote it by {$\{w^k:k\geq0\}$} updated by  {$w^{k+1} = w^k - \tau_k\nabla f(w^k)$}. Similar reasoning to show 
Lemma \ref{first-alg-finite} i) allows us to show that sequence $\{w^k:k\geq0\}\subseteq L(f(x^0))$ is bounded (i.e., $c:= \sup_{k\geq0}\|w^k\|<\infty$). Moreover, we note that $f$ is a continuously differentiable function and satisfies the Kurdyka-Lojasiewicz inequality. Therefore, according to Lemma \ref{gd-ami}, sequence {$\{w^k:k\geq0\}$}   converges to a stationary point $\bar{w}$ of $f$ which is not a strict saddle point. Hence we obtain $\nabla f(\bar{w})=0$ and $\nabla^2 f(\bar{w})\succeq0.$

Now consider a smaller radius $r=\sqrt{6/7}r_*$, thereby $\mathcal{N}_*(r)\subseteq \mathcal{N}_*$. The assumptions here reduce to the conditions in Proposition  \ref{strict-saddle}. Since $w^k\to\bar{w}$, there always is a sufficiently large integer $J<\infty$ such that  
\bea{llll}\sup_{k\geq J} \|\nabla f(w^{k})\|\leq  \epsilon(r),\qquad \sup_{k\geq J}  \|w^{k}-\bar{w}\|\leq \frac{(3{\lambda}-{\eta})\|x^*\|^2}{27{\eta}(c+\|\bar{w}\|)}.
\nnb\eda
  If $w^k\notin\mathcal{N}_*(r)$ for any $k\geq J$, then 
\bea{llll}
- {(3{\lambda}-{\eta})}\|x^*\|^2/8&\geq& \lambda_{\mathrm{min}}(\nabla^2f(w^k))&(\text{due to Proposition}~ \ref{strict-saddle})\\
&\geq&\lambda_{\mathrm{min}}(\nabla^2f(\bar{w}))-
 \|\nabla^2f(w^k)
-\nabla^2f(\bar{w})\|\\
&\geq& -
 \|\nabla^2f(w^k)
-\nabla^2f(\bar{w})\|~&(\text{due to}~\nabla^2 f(\bar{w})\succeq0)\\
&\geq& -3{\eta} 
\|w^k+\bar{w}\|\|w^k-\bar{w}\| ~&(\text{due to Lemma} ~\ref{hessian-lip})\\
&\geq& -3{\eta} (c+\|\bar{w}\|)\|w^k-\bar{w}\|\\
&\geq& - {(3{\lambda}-{\eta})}\|x^*\|^2/9,
\nnb\eda
which is a contradiction. Therefore, we have $w^k\in\mathcal{N}_*(r)$ for any $k\geq J$. The compactness of  set $\mathcal{N}_*(r)$ ensures $\bar{w}\in\mathcal{N}_*(r)\subseteq \mathcal{N}_*$. One can verify that \bea{llll}\kappa_*\leq \frac{(3{\lambda}-{\eta})r_*^2}{7} \leq \frac{2{\lambda} r_*^2}{7} \leq \frac{2{\lambda}}{7} \frac{\lambda^2\|x^*\|^2}{36 \eta^2} \leq   \frac{\lambda\|x^*\|^2}{2}   
\nnb\eda
due to $\lambda<\eta$. Hence, the first result of Lemma \ref{hessian-pd} implies that $f$ is a strong convex in $\mathcal{N}_*$, which indicates $\bar{w}$ is a unique local minimizer in $\mathcal{N}_*$.

By Theorem \ref{convergence},  whole sequence $\{x^k\}$ converges to a stationary point, denoted by $\bar{x}$. Therefore, given $r_*-r>0$, there is a finite integer $K_1>0$ such that
$$\|x^k-\bar{x}\|<(r_*-r)/2, ~~\forall k\geq K_1.$$
Similarly,  as $\{w^k\}$ converges to $\bar{w}$,  given $r_*-r>0$,  there is a finite integer $K_2>0$ such that
$$\|w^k-\bar{w}\|<(r_*-r)/2,~~\forall k\geq K_2.$$
Therefore, let $K_0=\max\{J, K_1,K_2\}$, then by $K\geq K_0$ we have
$$\|\bar{x}-\bar{w}\|\leq \|x^K-\bar{x}\|+\|x^K-\bar{w}\|= \|x^K-\bar{x}\|+\|w^K-\bar{w}\|<r_*-r,$$
where the equality is due to $x^k=w^k$ for any $k\in[K]$. The above fact contributes to
 \bea{lll}
 \min\{\|\bar{x}-x^*\|, \|\bar{x}+x^*\|\} &\leq& \min\{ \|\bar{x}-\bar{w}\| + \|\bar{w}-x^*\|, \|\bar{x}-\bar{w}\| + \|\bar{w}+x^*\|\}\\
 &=& \|\bar{x}-\bar{w}\| + \min\{  \|\bar{w}-x^*\|,   \|\bar{w}+x^*\|\} \leq r_*-r+r=r_*\nnb
 \eda
	due to $\bar{w}\in\mathcal{N}_*(r)$.  So $\bar{x}\in\mathcal{N}_*$. We note that $\bar{x}$ is a stationary point, and thus it is also a unique local minimizer in $\mathcal{N}_*$, thereby $\bar{x}=\bar{w}$.
\end{proof}}

\subsection{Analysis for noiseless model (\ref{PQMR})}
\begin{theorem}\label{oracle-noiseless}
Suppose that Assumption \ref{sample-RIP-general} holds and ${\eta}<3{\lambda}$.
Let $\{ {x}^{k}\}$ be the sequence generated by Algorithm \ref{algorithm 0} with setting a sufficiently large (but finite) $K_0$.
Then
 \begin{itemize}[leftmargin=20pt]
 \item[i)] the whole sequence converges to $x^*$ or $-x^*$, the true value of the noiseless model (\ref{PQMR}),
  \item[ii)]  and the convergence rate is $1+\delta$ if we set   $ \mu_2\in(0,1/2)$.
\end{itemize}
\end{theorem}
{
\begin{proof}
The first claim in Lemma \ref{hessian-pd} indicates $f$ is strongly convex in $\mathcal{N}_*$ and thus any point $\bar{x}\in\mathcal{N}_*$ such that $\nabla f(\bar{x})=0$ is the unique optimal solution to problem ${\rm min}_{x\in\mathcal{N}_*} f(x)$. In addition, $0\leq f(\bar{x})={\rm min}_{x\in\mathcal{N}_*} f(x)\leq f(x^*)=0$. Therefore, we have $\bar{x}=x^*$ or $\bar{x}=-x^*$. Then, the conclusion i) is a direct result of Lemma \ref{convergence-local}.  This together with  Theorem \ref{convergence-rate} results in conclusion ii) immediately as condition (\ref{hessian-min}) also holds.
\end{proof}
\noindent Based on Lemma \ref{sample-RIP}, we can set ${\eta}=(2+3\gamma )\sigma^2/2$ and ${\lambda}=(1-2\gamma)\sigma^2/2$ to ensure Assumption \ref{sample-RIP-general}, leading to ${\eta}<3{\lambda}$ if $\gamma\in(0,1/9)$. Therefore, a direct result of the above theorem and Lemma \ref{sample-RIP} is the following corollary.
\begin{corollary}\label{cor-noiseless}
Let $\{A_1,\ldots,A_n\}$ be generated as Setup \ref{ass-A-i} and $\gamma\in(0,1/9)$, and let $\{ {x}^{k}\}$ be the sequence generated by Algorithm \ref{algorithm 0} with setting a sufficiently large (but finite) $K_0$.  If condition (\ref{np-RIP}) holds, then with a high probability, then the whole sequence converges to the true value of noiseless model (\ref{PQMR}), and the convergence  rate is $1+\delta$ if we set   $ \mu_2\in(0,1/2)$.
\end{corollary}
}

\subsection{Analysis for noisy model (\ref{PQMR})}

{
\begin{theorem}\label{oracle-conv-noise}Suppose that Assumptions \ref{sample-RIP-general} and \ref{ass-error} hold and ${\eta}<3{\lambda}$.
Let $\{ {x}^{k}\}$ be the sequence generated by Algorithm \ref{algorithm 0} with setting a sufficiently large (but finite) $K_0$. If \bea{c}\label{nplowbound}\frac{n}{p\log(1+2n)}\geq  \frac{14^26^5\varsigma^2{\eta}^5}{(3{\lambda}-{\eta})^2{\lambda}^4\|x^*\|^4},\eda
 then with a probability at least $1-p_1-p_2$, the whole sequence converges to   $\bar{x}$ which is a global minimizer of (\ref{LSE-PQMR}) and satisfies
\bea{c}
\min\{\|\bar{x}-x^*\|,\|\bar{x}+x^*\|\}\leq\frac{4}{{\lambda}\|x^*\|}
\sqrt{\frac{6\varsigma    ^2{\eta}p\log(1+2n)}{n}}.
\nnb\eda
\end{theorem}}
\begin{proof}
{Recalling \eqref{def-N-r*} that $r_*=\lambda\|x^*\|/(6{\eta})$, we conclude from (\ref{nplowbound}) that \bea{lll}
2\sqrt{\frac{6\varsigma    ^2{\eta}p\log(1+2n)}{n}}
\leq\frac{3{\lambda}-{\eta}}{7}r_*^2,\\
\frac{4}{{\lambda}\|x^*\|}
\sqrt{\frac{6\varsigma    ^2{\eta}p\log(1+2n)}{n}} \leq \frac{2}{{\lambda}\|x^*\|} \frac{3{\lambda}-{\eta}}{7}r_*^2 = \frac{2}{{\lambda}\|x^*\|} \frac{3{\lambda}-{\eta}}{7} \frac{\lambda\|x^*\|}{6\eta} r_*\leq r_*.
 \nnb\eda By the second result of Lemma \ref{err-sup-ub}, we have
 \bea{lll}
\kappa_*\leq2\sqrt{\frac{6\varsigma    ^2{\eta}p\log(1+2n)}{n}}, \nnb\eda
with a probability at least $1-p_1-p_2$. This means that condition $\kappa_*\leq({3{\lambda}-{\eta}})r_*^2/7$ holds  with a probability at least $1-p_1-p_2$. Consequently,
Lemma \ref{convergence-local} implies that the sequence converges converges to a unique local minimizer $\bar{x}\in\mathcal{N}_*$. Recall Theorem \ref{nonasym-bound} that 
\bea{c}
\min\{\|\hat{x}-x^*\|,\|\hat{x}+x^*\|\}\leq\frac{4}{{\lambda}\|x^*\|}
\sqrt{\frac{6\varsigma    ^2{\eta}p\log(1+2n)}{n}}\leq r_*.
\nnb\eda
Therefore, $\hat{x}\in\mathcal{N}_*$. We note that there is only one local minimizer in $\mathcal{N}_*$, thereby $\bar{x}=\hat{x}$.}
\end{proof}

\begin{theorem}\label{oracle-conv-suplinear-noise}
Suppose that Assumptions \ref{sample-RIP-general} and \ref{ass-error} hold and ${\eta}<3{\lambda}$.
Let $\{ {x}^{k}\}$ be the sequence generated by Algorithm \ref{algorithm 0} with setting a sufficiently large (but finite) $K_0$ and $ \mu_2\in(0,1/2)$. If
\bea{c}\label{npto0}\frac{p^3\log(1+2n)}{n}\to0,\qquad\mbox{as}\qquad n\to\infty,\eda
then the whole sequence converges to $\bar{x}$ superlinearly in probability.
\end{theorem}
\begin{proof}
 It is easy to see that
\bea{lll}\|\frac{1}{n}\sum_{i=1}^n\varphi_i(\bar{x})A_i\|&\leq&\|\frac{1}{n}\sum_{i=1}^n\langle \bar{x}-x^*, A_i (\bar{x}+x^*) \rangle A_i\|+
\|\frac{1}{n}\sum_{i=1}^n\varphi_i(x^*) A_i\|.
\nnb\eda
 For the first term,
\bea{lll}\label{term-1}
\|\frac{1}{n}\sum_{i=1}^n\langle x^-, A_i x^+ \rangle A_i\|
&=&\sup_{\|u\|=1}|\frac{1}{n}\sum_{i=1}^n\langle x^-, A_i x^+ \rangle \langle u, A_iu\rangle|\\
&\leq&\sup_{\|u\|=1}\sqrt{\frac{1}{n}\sum_{i=1}^n\langle x^-, A_i x^+ \rangle^2}\sqrt{\frac{1}{n}\sum_{i=1}^n\langle u, A_iu\rangle^2}\\
&\leq&{\eta}\|x^-\|\|x^+\|.
\nnb\eda
Without loss of generality, we let $\|\bar{x}-x^*\|\leq\| \bar{x}+x^*\|$. Therefore, we use again the second result of Lemma \ref{err-sup-ub} to conclude that
\bea{lll}\label{n-phi-A}
\|\frac{1}{n}\sum_{i=1}^n\varphi_i(\bar{x})A_i\|&\leq&{\eta}\|x^-\|\|x^+\|
+2\sqrt{\frac{6\varsigma    ^2{\eta}p\log(1+2n)}{n}}\\
&\leq&{\eta}(2\|x^*\|+\|x^-\|)\|x^-\|
+2\sqrt{\frac{6\varsigma    ^2{\eta}p\log(1+2n)}{n}}\\
&\leq&2\sqrt{\frac{6\varsigma    ^2{\eta}p\log(1+2n)}{n}}\left[ {\eta}(2\|x^*\|+\|x^-\|)\frac{2}{{\lambda}\|x^*\|}
+1\right] ~~(\text{by Lemma \ref{oracle-conv-noise}})
\nnb\eda
with probability at least $1-p_1-p_2$.
Since \eqref{npto0} implies that $\sqrt{\frac{6\varsigma    ^2{\eta}p\log(1+2n)}{n}}\to0$ as $n\to\infty$, it follows that $\|\frac{1}{n}\sum_{i=1}^n\varphi_i(\bar{x})A_i\|$ can be sufficiently small. This  indicates \eqref{hessian-min} always holds for sufficiently large $n$ with probability at least $1-p_1-p_2$.  Then by Theorem \ref{convergence-rate} and  \eqref{def-c3}, we get
\bea{llll}
\lim_{k\to \infty} \frac{\|x^{k+1}-\bar{x}\|}{\|x^k-\bar{x}\|}
 &\leq& \lim_{k\to \infty} \left( c_1\| x^{k}-\bar{x}\| + c_2\|x^{k}-\bar{x}\| ^{\delta} +C\|\frac{1}{n}\sum_{i=1}^n \varphi_i(\bar{x})A_i\|\right)\\
 &=&  C\|\frac{1}{n}\sum_{i=1}^n \varphi_i(\bar{x})A_i\|,
 \nnb\eda
which together with $\|\frac{1}{n}\sum_{i=1}^n\varphi_i(\bar{x})A_i\|\to0$ with probability as $n\to\infty$  yields the desired result.
\end{proof}
\noindent { Finally, we obtain a similar result to Corollary \ref{cor-noiseless} from the above two theorems and Lemma \ref{sample-RIP}.
\begin{corollary}Let $\{A_1,\ldots,A_n\}$ be generated as Setup \ref{ass-A-i} and $\gamma\in(0,1/9)$. Suppose the errors in (\ref{PQMR}) satisfy Assumption \ref{ass-error} and are independent of $\{A_1,\ldots,A_n\}$. Let $\{ {x}^{k}\}$ be the sequence generated by Algorithm \ref{algorithm 0} with setting a sufficiently large (but finite) $K_0$ and $ \mu_2\in(0,1/2)$. If  (\ref{npto0})  holds, then the whole sequence converges to the true value of noisy model (\ref{PQMR}) with a high probability, and the convergence is superlinear in probability .
\end{corollary}}

\section{Numerical Examples}\label{sec:num}
In this section, we demonstrate the efficiency of our proposed algorithm GRNM.  In all experiments, the initial points for GRNM are chosen as ${x^0}/{f(0)}$, where the entries of $x^0\in\mathbb{R}^p$ or $x^0\in\mathbb{C}^p$ are generated from a standard Gaussian distribution. The involved parameters are set as follows. We choose  $\mu_1=0.1, \alpha_1=0.2$, $\epsilon_G=0.1p$, and $K_0=1$ to accelerate the computation for phase I and $\mu_2=0.1, \beta=0.5, \alpha_2=0.5,  \delta=0.25$ for phase II. Moreover, we terminate our algorithm if $\|g^k\|<10^{-5}$ in phase II or if the number of iterations reaches 5000.

We compare our method to the iterative
Wirtinger Flow (WF) method proposed by \cite{hgd19, hgd20}. To guarantee the best recovery performance of WF, we use its suggested spectral initial points for WF.
For  step size $\alpha$, we keep the default setting if $\sigma=1$ and adjust it as  $\alpha/\sigma$ otherwise, to avoid the increase of the objective function values. {Furthermore, as noted in a survey paper \cite{fann20}, some phase retrieval algorithms can be extended to natural generalizations of phase retrieval problems. For example, it has been shown in \cite{kueng17} that  Phaselift can be extended to low rank matrix recovery. Like in \cite{fann20}, we use the Phasepack toolbox \cite{chandra17} and adjust two kinds of algorithms, Phaselift and reweighted Wirtinger Flow (RWF), to the QMR. To demonstrate the performance of the four algorithms, we report the    CPU time (in second), relative error, and  success rate (defined by the rate of the number of successful recovery trials over the total trials). The relative error between $\hat{x}$ and $x^*$ is calculated by
\bea{l}
{\rm Relative~ error}  := \frac{\min_{\theta\in[0,2\pi]}\|x^*-\hat{x}e^{j\theta }\|}{\|x^*\|}
\eda
where $j$ is the  imaginary unit, $x^*$ is the  true signal, and $\hat{x}$ is the solution obtained by a solver.  We say that the recovery is successful if relative error $<10^{-5}$ for the noiseless case and relative error $ <5\times10^{-3}$ for the noisy case. }

\subsection{Real Gaussian measurements}
For each trial, the entries of true signal $x^*\in\mathbb{R}^p$ are generated from a standard Gaussian distribution and $n$ matrices ${B}_1,\ldots, {B}_n$ are  sampled independently and randomly with each entry of ${B}_i$ being an i.i.d. variable from $\mathcal{N}(0,\sigma^2)$, the normal distribution. Then we set
$A_i=( {B}_i+ {B}_i^T)/{2}$ and $b_i=\langle x^*, A_ix^*\rangle + \mathcal{N}(0,\sigma_\varepsilon^2)$ for $i\in[n]$.

{
\begin{table}[!h]
\renewcommand{\arraystretch}{1.0}\addtolength{\tabcolsep}{.5pt}
\begin{center}
\caption{Impact of two phases on GRNM. \label{table-GRNM}}
\begin{tabular}{lrcrrrcrrr} \hline
&\multicolumn{3}{c}{Phase I}&&\multicolumn{3}{c}{Phase II}&\\
$\epsilon_G$	&	Rel. error	&	Iterations	&	Time	&&	Rel. error	&	Iterations	&	Time	&&	Total Time	\\\cline{2-4}\cline{6-8}
&\multicolumn{9}{c}{$\sigma_\varepsilon=0$} \\\hline
$0.5p$	&	9.95e-1	&	2.0 	&	0.0102 	&&	2.47e-9	&	7.4 	&	0.0506 	&&	0.0608 	\\
$0.1p$	&	7.50e-1	&	4.6 	&	0.0252 	&&	1.61e-9	&	6.6 	&	0.0472 	&&	0.0724 	\\
$0.05p$	&	6.95e-3	&	17.4 	&	0.0978 	&&	1.37e-9	&	3.0 	&	0.0143 	&&	0.1121 	\\
$0.01p$	&	1.13e-3	&	23.0 	&	0.1270 	&&	1.75e-9	&	2.0 	&	0.0103 	&&	0.1373 	\\
$0.005p$	&	7.53e-4	&	25.0 	&	0.1397 	&&	1.65e-9	&	2.0 	&	0.0101 	&&	0.1499 	\\
$0.0001p$	&	2.43e-5	&	36.9 	&	0.1915 	&&	1.33e-8	&	1.0 	&	0.0063 	&&	0.1977 	\\ \hline
&\multicolumn{9}{c}{$\sigma_\varepsilon=0.1$} \\\hline
$0.5p$	&	9.97e-1	&	2.0 	&	0.0104 	&&	4.22e-4	&	7.5 	&	0.0561 	&&	0.0665 	\\
$0.1p$	&	8.76e-1	&	3.3 	&	0.0304 	&&	4.22e-4	&	7.0 	&	0.0525 	&&	0.0829 	\\
$0.05p$	&	5.92e-3	&	17.5 	&	0.0985 	&&	4.22e-4	&	3.0 	&	0.0136 	&&	0.1121 	\\
$0.01p$	&	1.29e-3	&	24.0 	&	0.1298 	&&	4.22e-4	&	2.1 	&	0.0114 	&&	0.1412 	\\
$0.005p$	&	7.97e-4	&	26.9 	&	0.1471 	&&	4.22e-4	&	2.0 	&	0.0090 	&&	0.1560 	\\
$0.0001p$	&	4.22e-4	&	37.8 	&	0.1994 	&&	4.22e-4	&	1.1 	&	0.0064 	&&	0.2058 	\\

\hline
 \end{tabular}
\end{center}
\end{table}

First, we examine the impact of two phases on GRNM by selecting different values of $\epsilon_G$ and running 100 trials to report the average results. For simplicity, we set $p = 100$, $n = 4p$, and $\sigma = 1$. As shown in Table \ref{table-GRNM}, $\sigma_\varepsilon = 0$ and $0.1$ represent the noiseless and noisy models of (\ref{PQMR}), respectively. It is evident that a smaller $\epsilon_G$ results in higher relative error and fewer iterations in phase I but more iterations in phase II. However, there is little difference in the final relative error (namely, the ones in phase II). Notably, a smaller $\epsilon_G$ leads to shorter total computational time. Therefore, we set $\epsilon_G = 0.1p$ to accelerate computation in the subsequent numerical experiments.

\begin{figure}[!t]
\centering
\begin{subfigure}{.495\linewidth}\centering
\includegraphics[scale=.75]{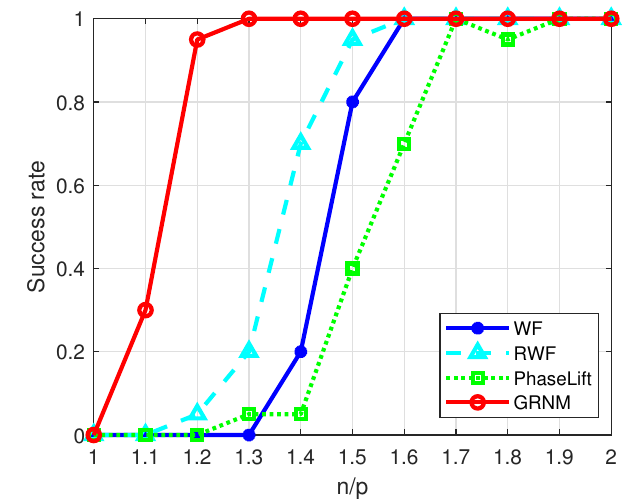} \caption{$\sigma_\varepsilon=0$}\label{fig-effect-n-noiseless}
\end{subfigure}~~
\begin{subfigure}{.495\linewidth}\centering
\includegraphics[scale=.75]{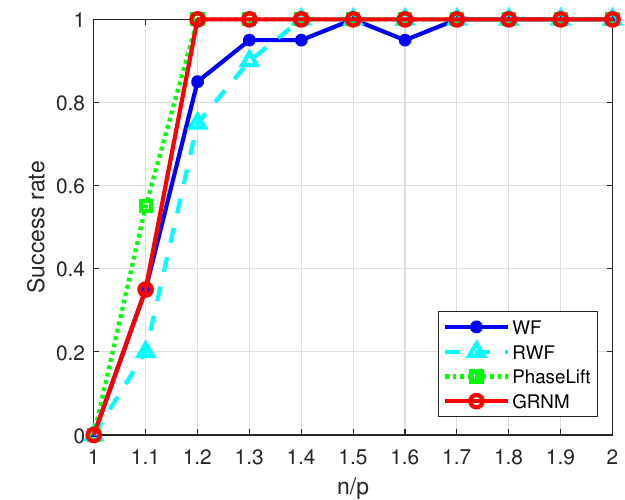}
\caption{$\sigma_\varepsilon=0.1$} \label{fig-effect-n-noise}
\end{subfigure}
\caption{Success rate v.s. $n/p$ for real Gaussian measurements.}\label{fig-effect-n}
\end{figure}%
\vspace{-2mm}

\begin{figure}[!t]
\centering
\begin{subfigure}{.495\linewidth}\centering
\includegraphics[scale=0.75]{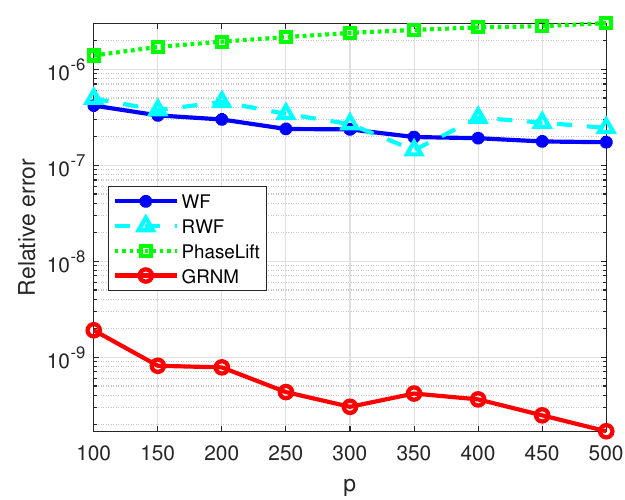}
\caption{$\sigma_\varepsilon=0$}\label{fig-effect-p-relerr-noiseless}
\end{subfigure}~~~
\begin{subfigure}{.495\linewidth}
\includegraphics[scale=0.75]{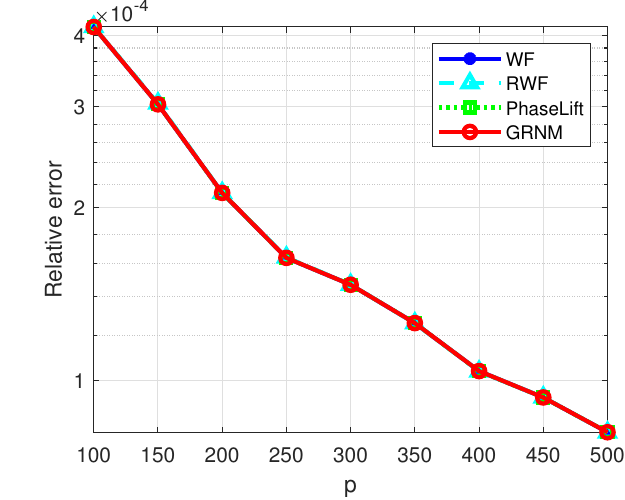}
\caption{$\sigma_\varepsilon=0.1$}\label{fig-effect-p-relerr-noise}
\end{subfigure}\\ 
\begin{subfigure}{.495\linewidth}
\includegraphics[scale=0.75]{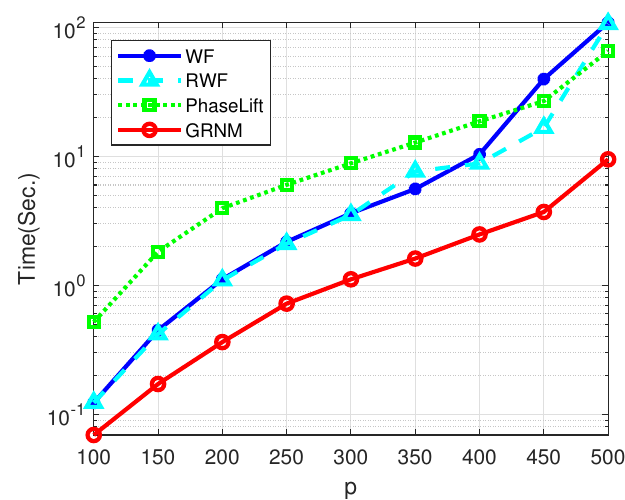}
\caption{$\sigma_\varepsilon=0$}\label{fig-effect-p-time-noiseless}
\end{subfigure}~~~
\begin{subfigure}{.495\linewidth}
\includegraphics[scale=0.75]{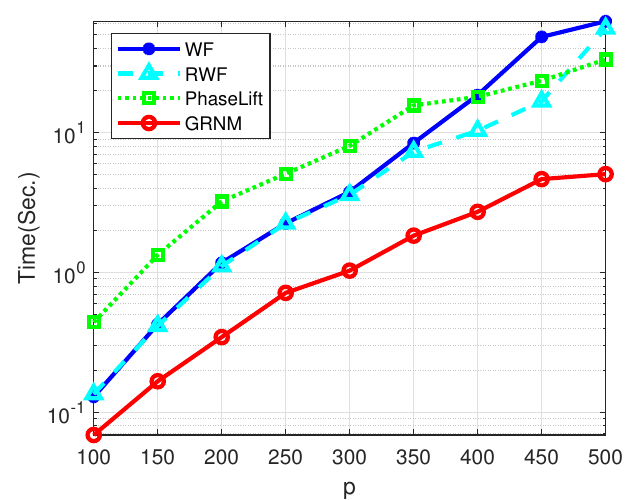}
\caption{$\sigma_\varepsilon=0.1$}\label{fig-effect-p-time-noise}
\end{subfigure}
\caption{Relative error and CPU time v.s. $p$ for real Gaussian measurements.}\label{fig-effect-p}
\end{figure}%

\textbf{(a) Effect of ${n}/{p}$.} We fix $\sigma=1$ and $ p=100$ but alter ${n}/{p}\in\{1,1.1,\ldots,1.9,2\}$ and $\sigma_\varepsilon\in\{0,0.1\}$.  For each combination of $(p,n,\sigma,\sigma_\varepsilon)$, we run 100 trials and report the success rate in Figure \ref{fig-effect-n}. It can be clearly seen that GRNM achieves a much higher success rate than the other three algorithms for the noiseless cases and slightly better results for the noisy cases.

\textbf{(b) Effect of $p$}. We fix $\sigma=1$ and $ n=4p$, but vary $p \in \{100,150,\ldots,450,500\}$ and $\sigma_\varepsilon\in\{0,0.1\}$. We run 25 random trials for each combination of $(p,n,\sigma,\sigma_\varepsilon)$ as WF takes too long time when $p\geq 450$. Since all methods have $100\%$ success rates when $n=4p$ as shown in Figure \ref{fig-effect-n}, we report the relative error and CPU time in Figure \ref{fig-effect-p}. Evidently, GRNM achieves a much higher accuracy than the other three algorithms for the noiseless cases and runs much faster for both cases.

\textbf{(c) Effect of $\sigma$}. We fix $p=100$ and $n=400$, but vary  $\sigma_\varepsilon\in\{0,0.1\}$ and $\sigma\in\{1,2,\ldots,10\}$ to see the effect of $\sigma$. The average results over 100 trials are reported in Table  \ref{table-rA}. Apparently, our method GRMN outperforms the others. In particular, it runs the fastest.

\begin{table}[th]
\renewcommand{\arraystretch}{1.1}\addtolength{\tabcolsep}{3pt}
\begin{center}
\caption{Effect of $\sigma$ for real Gaussian measurements. \label{table-rA}}
\begin{tabular}{c c c c c c c c c c} \hline
   &  \multicolumn{4}{c}{Relative error}
 &  &  \multicolumn{4}{c}{CPU Time} \\
 $\sigma$ & GRMN    & WF      & RWF     & Phaselift  &&GRMN      & WF      & RWF      & Phaselift  \\\cline{2-5}\cline{7-10}
 	&    \multicolumn{9}{c}{$\sigma_\varepsilon=0$} \\\hline
1         &2.10e-9  &1.18e-6  &5.46e-7  &1.42e-6     &&0.062     &0.413    &0.275     &0.619\\
2         &4.78e-10 &1.14e-6  &4.87e-7  &1.01e-6     &&0.066     &0.459    &0.312     &0.735\\
3         &1.26e-10 &1.15e-6  &5.41e-7  &8.11e-7     &&0.060     &0.441    &0.248     &0.670\\
4         &1.13e-10 &1.14e-6  &5.30e-7  &7.04e-7     &&0.058     &0.437    &0.253     &0.664\\
5         &5.55e-11 &1.15e-6  &4.33e-7  &6.39e-7     &&0.062     &0.450    &0.247     &0.667\\
6         &4.81e-11 &1.14e-6  &5.59e-7  &5.74e-7     &&0.059     &0.449    &0.250     &0.668\\
7         &2.44e-11 &1.15e-6  &4.46e-7  &5.30e-7     &&0.064     &0.454    &0.254     &0.667\\
8         &2.64e-11 &1.57e-6  &4.83e-7  &5.03e-7     &&0.072     &0.464    &0.273     &0.669\\
9         &1.15e-11 &1.16e-6  &5.12e-7  &4.77e-7     &&0.066     &0.456    &0.266     &0.666\\
10        &1.68e-11 &1.16e-6  &4.46e-7  &4.50e-7     &&0.065     &0.459    &0.275     &0.666\\
\hline
	&    \multicolumn{9}{c}{$\sigma_\varepsilon=0.1$} \\\hline
1         &4.33e-4  &4.33e-4  &4.32e-4  &4.33e-4     &&0.080     &0.470    &0.208     &0.456\\
2         &2.19e-4  &2.19e-4  &2.19e-4  &2.19e-4     &&0.072     &0.486    &0.200     &0.431\\
3         &1.31e-4  &1.31e-4  &1.31e-4  &1.31e-4     &&0.063     &0.461    &0.225     &0.490\\
4         &9.40e-5  &9.39e-5  &9.41e-5  &9.40e-5     &&0.054     &0.438    &0.234     &0.543\\
5         &7.97e-5  &7.97e-5  &7.98e-5  &7.97e-5     &&0.062     &0.466    &0.243     &0.541\\
6         &5.97e-5  &5.98e-5  &5.98e-5  &5.97e-5     &&0.071     &0.505    &0.267     &0.582\\
7         &6.03e-5  &6.02e-5  &6.02e-5  &6.03e-5     &&0.068     &0.484    &0.257     &0.565\\
8         &5.16e-5  &5.15e-5  &5.15e-5  &5.16e-5     &&0.055     &0.452    &0.232     &0.573\\
9         &4.62e-5  &4.62e-5  &4.64e-5  &4.62e-5     &&0.058     &0.505    &0.291     &0.561\\
10        &4.20e-5  &4.20e-5  &4.19e-5  &4.20e-5     &&0.071     &0.480    &0.264     &0.551\\
\hline
 \end{tabular}
\end{center}
\end{table}
\vspace{-3mm}

\begin{figure}[!th]
\centering
\begin{subfigure}{.49\linewidth}
\includegraphics[scale=0.75]{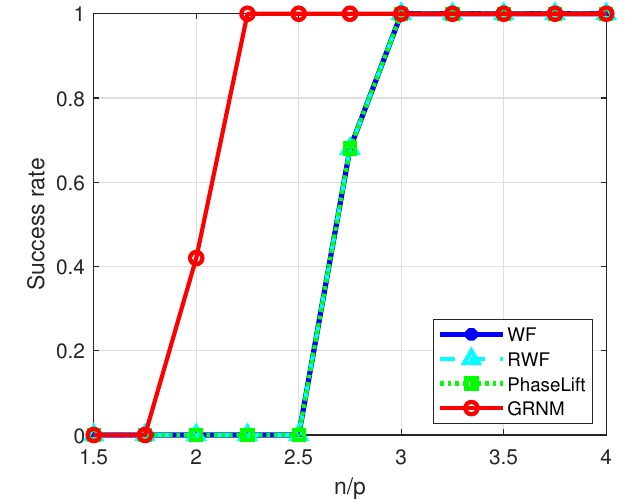}
\caption{$\sigma_\varepsilon=0$}
\end{subfigure}
\begin{subfigure}{.49\linewidth}
\includegraphics[scale=0.75]{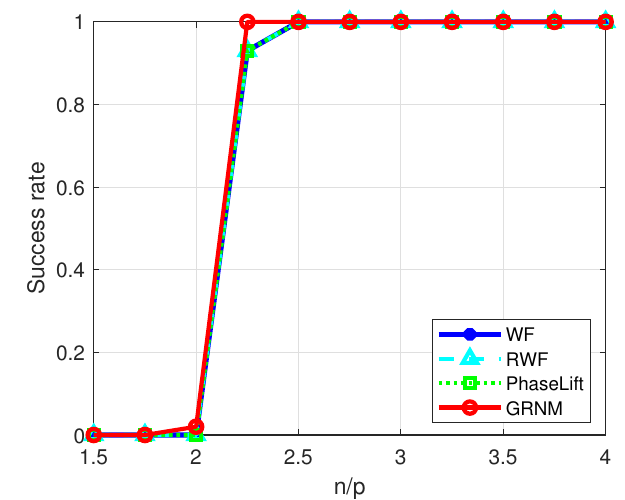}
\caption{$\sigma_\varepsilon=0.1$}
\end{subfigure}
\caption{Success rate v.s. $n/p$ for complex Gaussian measurements.}\label{fig-effect-n-complex}
\end{figure}%

\begin{figure}[!h]
\centering
\begin{subfigure}{.495\linewidth}
\includegraphics[scale=0.75]{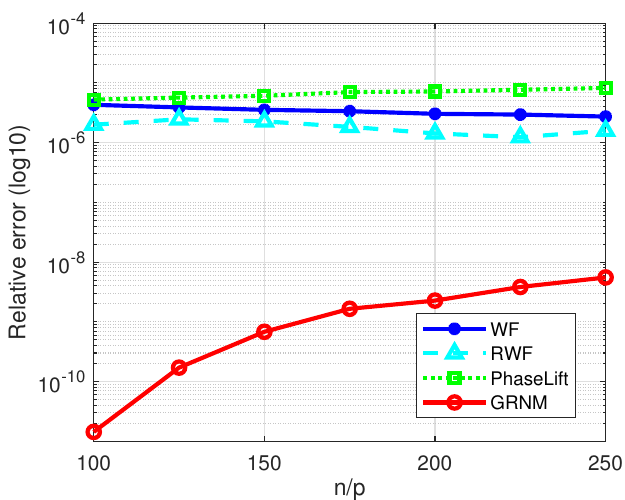}
\caption{$\sigma_\varepsilon=0$}
\end{subfigure}
\begin{subfigure}{.495\linewidth}
\includegraphics[scale=0.75]{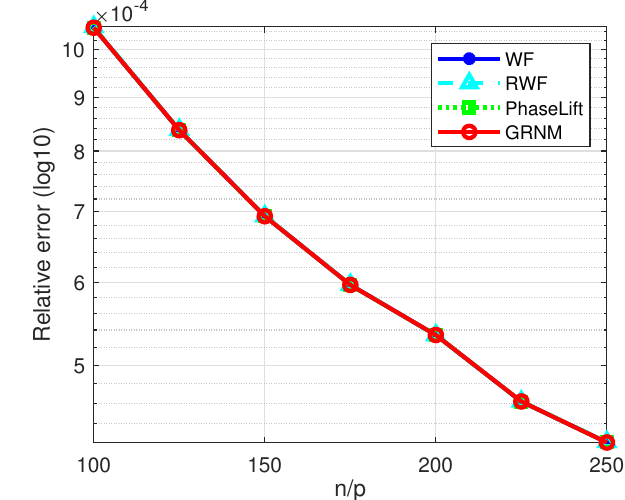}
\caption{$\sigma_\varepsilon=0.1$}
\end{subfigure}\\
\begin{subfigure}{.495\linewidth}
\includegraphics[scale=0.75]{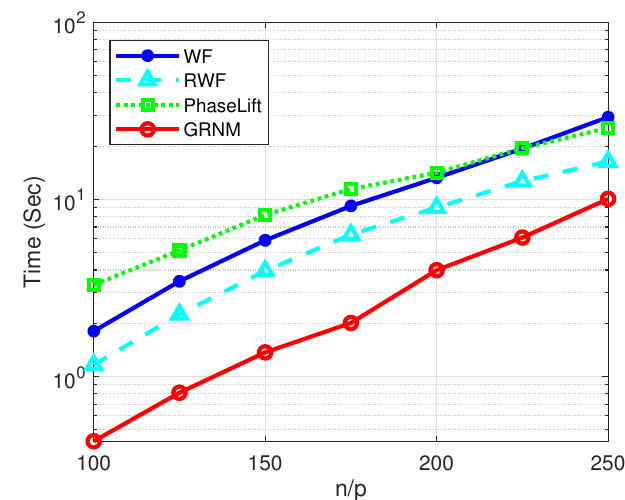}
\caption{$\sigma_\varepsilon=0$}
\end{subfigure}
\begin{subfigure}{.495\linewidth}
\includegraphics[scale=0.75]{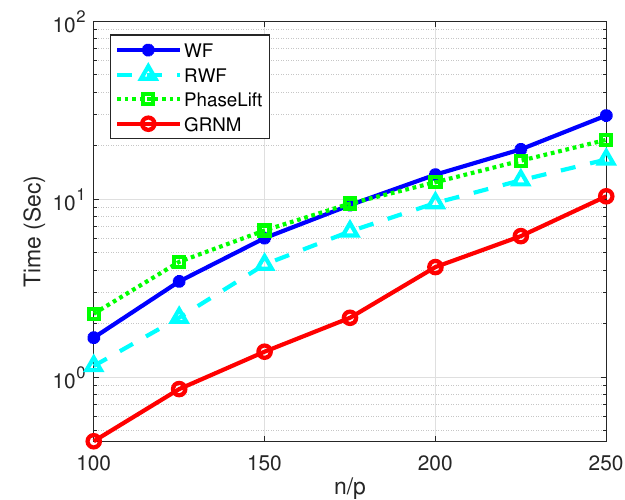}
\caption{$\sigma_\varepsilon=0.1$}
\end{subfigure}
\caption{Relative error and CPU time v.s. $p$ for complex Gaussian measurements.}\label{fig-effect-p-complex}
\end{figure}%

\subsection{Complex Gaussian measurements}\label{subsec:ccm}
For each trial, we first generate a random complex
signal $x^*\in\mathbb{C}^p$ with real and imaginary coefficients from $\mathcal{N}(0,1)$, followed by normalization to have a unit length, and then generate $2n$ random real Gaussian matrices ${R}_1,\ldots,{R}_n$, ${I}_1,\ldots,{I}_n$ independently, with  each entry of ${R}_i$ and ${I}_i$ being an i.i.d. variable from $\mathcal{N}(0,\sigma^2)$. Finally, we set
$A_i=({R}_i+ {R}_i^T+j( {I}_i- {I}_i^T))/{2}$ and $b_i=\langle x^*, A_ix^*\rangle +\mathcal{N}(0,\sigma_\varepsilon^2)$ for $i\in[n]$.


{\bf (d) Effect of $n/p$.} To examine the effect of $n/p$, we fix $\sigma=1$ and  $p=100$, but alter ${n}/{p}\in\{1.5,1.75,\ldots,3.75,4\}$ and $\sigma_\varepsilon\in\{0,0.1\}$. For each combination of  $(p,n,\sigma,\sigma_\varepsilon)$, we run 100 trials and report the success rate in Figure \ref{fig-effect-n-complex}, where the curves show that GRNM produces a much better success rate than the other three methods for the noiseless cases and yields slightly better results for the noisy cases.

{\bf (e) Effect of $p$.} Next, we fix $\sigma=1$ and $n=4p$ but vary  $p \in \{100,125,\ldots,225,250\}$ and $\sigma_\varepsilon\in\{0,0.1\}$ to see the effect of $p$. Average results over $25$ trials are presented in Figure \ref{fig-effect-p-complex},  where GRNM achieves the lowest relative errors (i.e., highest order of accuracy) for the noiseless cases and runs the fastest in all cases.

 {\bf (f) Effect of $\sigma$.}  Similarly, we fix $p=100$ and $n=400$ but choose $\sigma_\varepsilon\in\{0,0.1\}$ and $\sigma\in\{1,2,\ldots,10\}$. The average results over 100 trials are reported in Table \ref{table-cgau}. Once again, our method outperforms the others due to the highest accuracy and the fastest computational speed.

\begin{table}[th]
\renewcommand{\arraystretch}{1.1}\addtolength{\tabcolsep}{3pt}
\begin{center}
\caption{Effect of $\sigma$ for complex Gaussian measurements. \label{table-cgau}}
\begin{tabular}{c c c c c c c c c c} \hline
	& \multicolumn{4}{c}{Relative error}
	&  & \multicolumn{4}{c}{CPU Time} \\
	$\sigma$ & GRMN    & WF      & RWF     & Phaselift  &&GRMN      & WF      & RWF      & Phaselift  \\\cline{2-5}\cline{7-10}
		&    \multicolumn{9}{c}{$\sigma_\varepsilon=0$} \\\hline
	1         &2.49e-9  &4.37e-6  &2.02e-6  &5.24e-6     &&0.412     &2.463    &1.065     &2.829\\
	2         &1.39e-9  &4.43e-6  &2.04e-6  &3.58e-6     &&0.440     &2.474    &1.318     &3.363\\
	3         &1.83e-9  &4.27e-6  &2.76e-6  &2.99e-6     &&0.421     &2.526    &1.462     &3.315\\
	4         &2.49e-10 &4.26e-6  &1.94e-6  &2.56e-6     &&0.414     &2.649    &1.455     &3.004\\
	5         &3.12e-10 &4.43e-6  &2.88e-6  &2.30e-6     &&0.405     &2.477    &1.473     &3.017\\
	6         &1.95e-10 &4.13e-6  &2.62e-6  &2.01e-6     &&0.406     &2.478    &1.631     &3.070\\
	7         &2.50e-10 &4.32e-6  &2.43e-6  &2.01e-6     &&0.401     &2.509    &1.678     &2.996\\
	8         &1.09e-10 &4.43e-6  &1.99e-6  &1.73e-6     &&0.414     &2.515    &1.702     &3.098\\
	9         &1.07e-10 &4.25e-6  &2.02e-6  &1.66e-6     &&0.410     &2.509    &1.804     &3.007\\
	10        &8.36e-11 &4.44e-6  &2.65e-6  &1.69e-6     &&0.410     &2.498    &2.479     &2.979\\
	\hline
	&    \multicolumn{9}{c}{$\sigma_\varepsilon=0.1$} \\\hline
	1         &1.04e-3  &1.04e-3  &1.04e-3  &1.04e-3     &&0.436     &2.384    &1.254     &2.409\\
	2         &5.27e-4  &5.27e-4  &5.27e-4  &5.27e-4     &&0.424     &2.221    &1.255     &2.431\\
	3         &3.50e-4  &3.50e-4  &3.50e-4  &3.50e-4     &&0.420     &2.245    &1.317     &2.442\\
	4         &2.72e-5  &2.72e-4  &2.72e-4  &2.72e-4     &&0.427     &2.270    &1.408     &2.568\\
	5         &2.04e-4  &2.04e-4  &2.04e-4  &2.04e-4     &&0.420     &2.228    &1.402     &2.660\\
	6         &1.72e-4  &1.73e-4  &1.72e-4  &1.72e-4     &&0.420     &2.270    &1.544     &2.454\\
	7         &1.48e-4  &1.49e-4  &1.49e-4  &1.49e-4     &&0.417     &2.214    &1.648     &2.470\\
	8         &1.33e-4  &1.33e-4  &1.33e-4  &1.33e-4     &&0.417     &2.290    &1.678     &2.543\\
	9         &1.18e-4  &1.17e-4  &1.18e-4  &1.18e-4     &&0.418     &2.346    &1.707     &2.573\\
	10        &1.04e-4  &1.04e-4  &1.04e-4  &1.04e-4     &&0.420     &2.459    &1.946     &2.532\\
	\hline
\end{tabular}
\end{center}
\end{table}

\begin{figure}[!th]
\centering
\begin{subfigure}{.49\linewidth}
\includegraphics[scale=0.75]{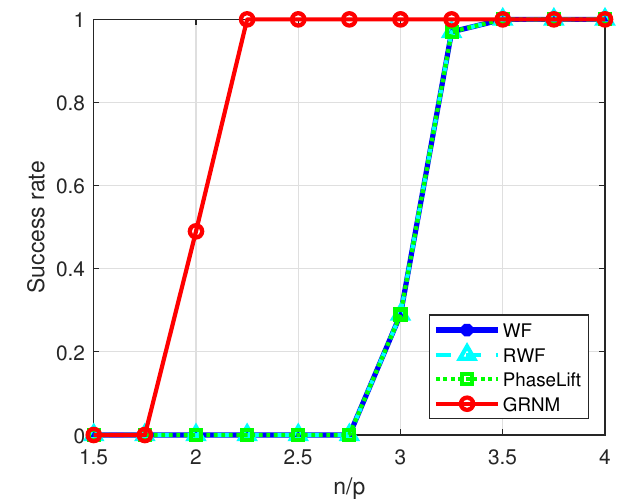}
\caption{$\sigma_\varepsilon=0$}
\end{subfigure}
\begin{subfigure}{.49\linewidth}
\includegraphics[scale=0.75]{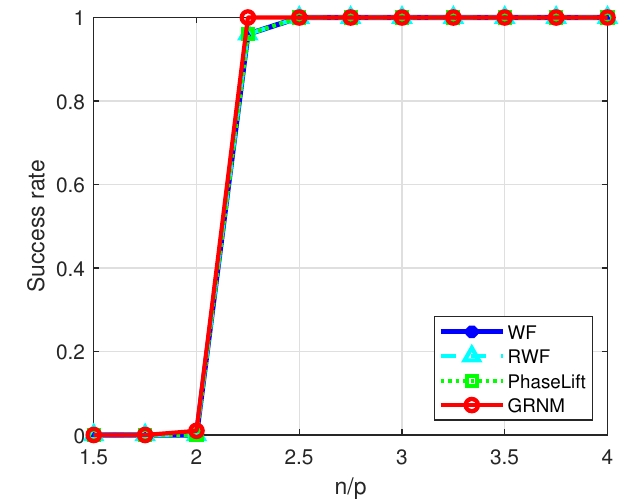}
\caption{$\sigma_\varepsilon=0.1$}
\end{subfigure}
\caption{Success rate v.s. $n/p$ for complex sub-Gaussian measurements.}\label{fig-effect-n-complex-sub}
\end{figure}%

\begin{figure}[!h]
\centering
\begin{subfigure}{.495\linewidth}
\includegraphics[scale=0.75]{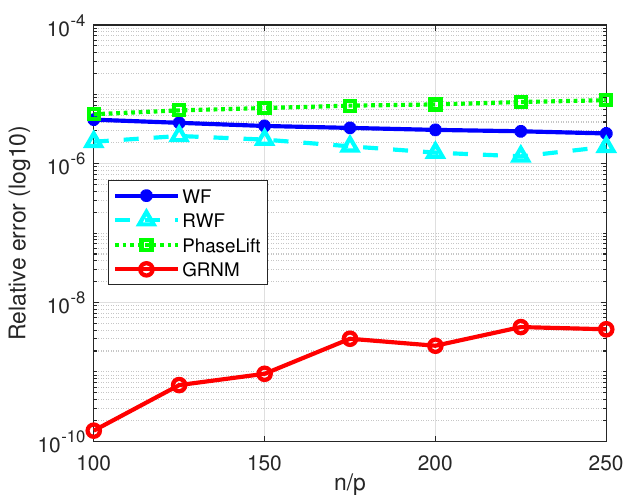}
\caption{$\sigma_\varepsilon=0$}
\end{subfigure}
\begin{subfigure}{.495\linewidth}
\includegraphics[scale=0.75]{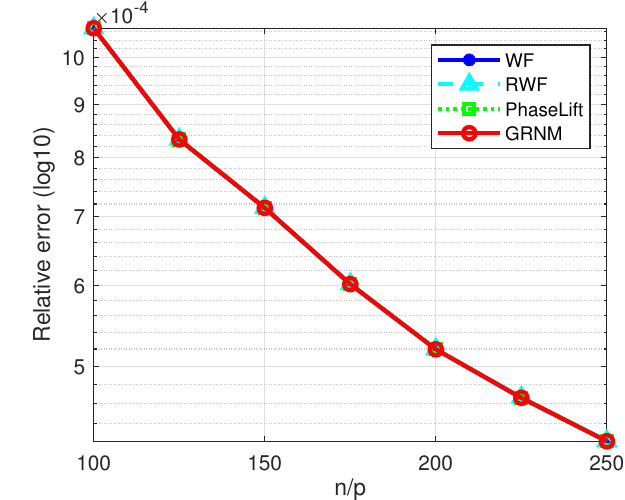}
\caption{$\sigma_\varepsilon=0.1$}
\end{subfigure}\\
\begin{subfigure}{.495\linewidth}
\includegraphics[scale=0.75]{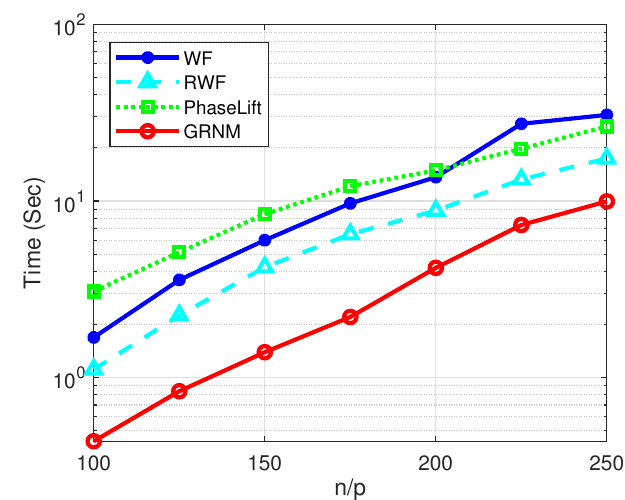}
\caption{$\sigma_\varepsilon=0$}
\end{subfigure}
\begin{subfigure}{.495\linewidth}
\includegraphics[scale=0.75]{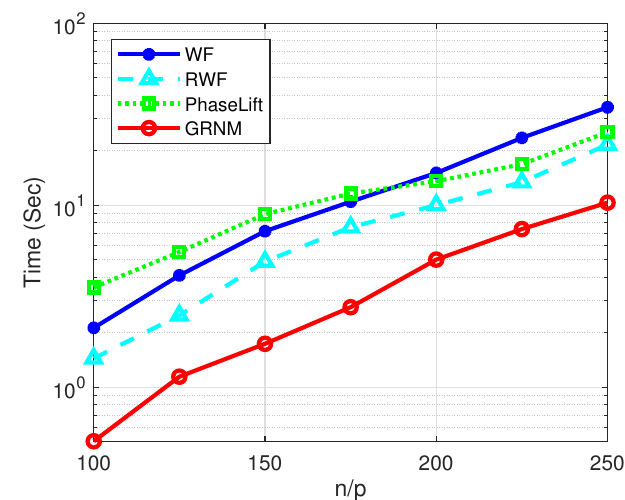}
\caption{$\sigma_\varepsilon=0.1$}
\end{subfigure}
\caption{Relative error and CPU time v.s. $p$ for complex sub-Gaussian measurements.}\label{fig-effect-p-complex-sub}
\end{figure}%

\subsection{Complex sub-Gaussian measurements}
We first generate a random complex signal $x^*\in\mathbb{C}^p$ as described in Section \ref{subsec:ccm} and $n$ complex random rotation-invariant sub-Gaussian matrices, $B_1,\ldots,B_n$ as used in \cite{hgd20}. We then set $A_i=(\sigma B_i+\sigma B_i^{H})/2$ and $b_i=\langle x^*, A_ix^*\rangle +\mathcal{N}(0,\sigma_\varepsilon^2)$ for $i\in[n]$.

{\bf (g) Effect of $n/p$.} To see this effect, we fix $\sigma=1$ and  $p=100$ but select $\sigma_\varepsilon\in\{0,0.1\}$ and  ${n}/{p}\in\{1.5,1.75,\ldots,3.75,4\}$.  The average results over 100 trials are shown in Figure \ref{fig-effect-n-complex-sub}. The observations are similar to those in Figure \ref{fig-effect-n-complex}, namely, GRNM significantly outperforms the others for the noiseless cases and shows slightly better results for the noisy cases.

{\bf (h) Effect of $p$.} To proceed with that, we fix $\sigma=1, n=4p$ but vary $\sigma_\varepsilon\in\{0,0.1\}$ and $p \in \{100,125,\ldots,225,250\}$.  The average results over 25 trials are displayed in Figure \ref{fig-effect-p-complex-sub}. It is clear that GRNM delivers a relatively high order of accuracy for the noiseless cases and consumes the shortest computational time for all scenarios.

{\bf (i) Effect of $\sigma$.}  Finally, we fix $p=100$ and $n=400$ but choose $\sigma\in\{1,2,\ldots,10\}$ and $\sigma_\varepsilon\in\{0,0.1\}$.  The average results over 100 trials are reported in Table \ref{table-csubgau}. The observations are similar to those in Table \ref{table-cgau}, namely, GRMN obtains the most accurate solutions and runs the fastest.}

\begin{table}[th]
\renewcommand{\arraystretch}{1.1}\addtolength{\tabcolsep}{3pt}
\begin{center}
\caption{Effect of $\sigma$ for complex sub-Gaussian measurements. \label{table-csubgau}}
\begin{tabular}{c c c c c c c c c c} \hline
	& \multicolumn{4}{c}{Relative error}
	&  & \multicolumn{4}{c}{CPU Time} \\
	$\sigma$ & GRMN    & WF      & RWF     & Phaselift  &&GRMN      & WF      & RWF      & Phaselift  \\\cline{2-5}\cline{7-10}
		&    \multicolumn{9}{c}{$\sigma_\varepsilon=0$} \\\hline
	1         &1.95e-9  &4.45e-6  &2.14e-6  &5.24e-6     &&0.465     &2.421    &1.226     &3.188\\
	2         &1.47e-9  &4.44e-6  &2.09e-6  &3.76e-6     &&0.409     &2.187    &1.190     &2.981\\
	3         &1.18e-9  &4.49e-6  &2.81e-6  &3.04e-6     &&0.405     &2.264    &1.298     &2.936\\
	4         &8.28e-10 &4.22e-6  &2.01e-6  &2.67e-6     &&0.398     &2.292    &1.325     &2.850\\
	5         &2.75e-10 &4.24e-6  &2.89e-6  &2.25e-6     &&0.401     &2.346    &1.441     &2.937\\
	6         &1.55e-10 &4.27e-6  &2.63e-6  &2.08e-6     &&0.405     &2.198    &1.478     &2.995\\
	7         &7.78e-11 &4.36e-6  &2.34e-6  &1.86e-6     &&0.401     &2.368    &1.639     &3.044\\
	8         &6.26e-11 &4.50e-6  &2.07e-6  &1.82e-6     &&0.411     &2.278    &1.553     &2.876\\
	9         &9.54e-11 &4.39e-6  &2.10e-6  &1.73e-6     &&0.414     &2.349    &1.765     &3.349\\
	10        &2.26e-10 &4.21e-6  &2.36e-6  &1.66e-6     &&0.412     &2.479    &1.919     &3.069\\
	\hline
	&    \multicolumn{9}{c}{$\sigma_\varepsilon=0.1$} \\\hline
	1         &1.06e-3  &1.06e-3  &1.06e-3  &1.06e-3     &&0.402     &2.037    &1.059     &2.098\\
	2         &5.12e-4  &5.13e-4  &5.13e-4  &5.13e-4     &&0.416     &2.211    &1.225     &2.281\\
	3         &3.52e-4  &3.52e-4  &3.53e-4  &3.52e-4     &&0.447     &2.527    &1.554     &3.028\\
	4         &2.57e-5  &2.57e-4  &2.57e-4  &2.57e-4     &&0.413     &2.397    &1.430     &2.561\\
	5         &2.06e-4  &2.07e-4  &2.07e-4  &2.07e-4     &&0.445     &2.596    &1.639     &2.705\\
	6         &1.74e-4  &1.74e-4  &1.74e-4  &1.74e-4     &&0.409     &2.261    &1.523     &2.558\\
	7         &1.54e-4  &1.54e-4  &1.54e-4  &1.54e-4     &&0.418     &2.268    &1.570     &2.499\\
	8         &1.33e-4  &1.33e-4  &1.33e-4  &1.33e-4     &&0.46      &2.315    &1.797     &2.657\\
	9         &1.16e-4  &1.16e-4  &1.16e-4  &1.16e-4     &&0.442     &2.502    &1.952     &3.204\\
	10        &1.07e-4  &1.07e-4  &1.07e-4  &1.07e-4     &&0.425     &2.437    &1.914     &2.625\\
	\hline
\end{tabular}
\end{center}
\end{table}

\section{Conclusion}
Our main theory has shown that the least squares model is capable of recovering the true signal from quadratic measurements as long as they are sub-Gaussian rather than Gaussian, regardless of the presence of noise.  The model was effectively solved using a gradient regularized Newton method, which is proven to converge globally and superlinearly. This leads to the freedom from dependence on initial points and achieves fast computational speed. Finally, the efficiency of the proposed algorithm was testified in both theoretical and numerical perspectives.

\vskip 14pt
\noindent {\large\bf Acknowledgements}
The research was supported by
  the National Natural Science Foundation of China (11801130), the Fundamental Research Funds for the Central Universities, the Talent Fund of Beijing Jiaotong University, and the Natural Science Foundation of Hebei Province
(A2019202135).
\par

\end{document}